\documentclass[onecolumn,11pt,draftcls]{IEEEtran}
\usepackage{amsbsy}
\usepackage{floatflt} 

\usepackage{amsmath}
\usepackage{amssymb}
\usepackage{times}
\usepackage{graphicx}
\usepackage{xspace}
\usepackage{paralist} 
\usepackage{setspace} 
\usepackage{xypic}
\xyoption{curve}
\usepackage{latexsym}
\usepackage{theorem}
\usepackage{ifthen}
\usepackage{stackrel}
\usepackage{subfigure}
\usepackage{booktabs}

\newcommand{\mypar}[1]{\vspace{0.03in}\noindent{\bf #1.}}

\newcounter{brojac}

{\theoremheaderfont{\it} \theorembodyfont{\rmfamily}
\newtheorem{assumption}[brojac]{Assumption}
}

{\theoremheaderfont{\it} \theorembodyfont{\rmfamily}
\newtheorem{theorem}{Theorem}
\newtheorem{lemma}[theorem]{Lemma}
\newtheorem{definition}[theorem]{Definition}
\newtheorem{example}[theorem]{Example}

\newtheorem{corollary}[theorem]{Corollary}
\newtheorem{observation}[theorem]{Observation}
}

\title{Consensus and Products of Random Stochastic Matrices: Exact Rate for Convergence in Probability}
\author{Dragana Bajovi\'c,  Jo\~ao
Xavier, Jos\'e M.~F.~Moura and Bruno Sinopoli  
\thanks{Work of Dragana Bajovi\'c, Jo\~ao
Xavier and Bruno Sinopoli is partially supported by grants CMU-PT/SIA/0026/2009 and
SFRH/BD/33517/2008 (through the Carnegie Mellon/Portugal Program
managed by ICTI) from Funda\c{c}\~{a}o para a Ci\^encia e Tecnologia and
also by
ISR/IST plurianual funding (POSC program, FEDER). Work of Jos\'e~M.~F.~Moura is partially supported by NSF under grants CCF-1011903 and CCF-1018509, and by AFOSR grant
FA95501010291. Dragana Bajovi\'{c} holds fellowship
from the Carnegie Mellon/Portugal Program.}
\thanks{Dragana Bajovi\'{c} is with the
Institute for Systems and Robotics
(ISR), Instituto Superior T\'{e}cnico (IST), Lisbon, Portugal, and with
the Department of Electrical and Computer Engineering, Carnegie Mellon
University, Pittsburgh, PA, USA {\tt\small dragana@isr.ist.utl.pt, dbajovic@andrew.cmu.edu}}%
\thanks{Jo\~ao Xavier is with the Institute for Systems and Robotics (ISR),
Instituto Superior T\'{e}cnico (IST), Lisbon, Portugal {\tt\small
jxavier@isr.ist.utl.pt}}
\thanks{Bruno Sinopoli and Jos\'e M.~F.~Moura are with the Department of
Electrical and Computer
Engineering, Carnegie Mellon University, Pittsburgh, PA, USA {\tt\small
brunos@ece.cmu.edu, moura@ece.cmu.edu; ph: (412)268-6341; fax: (412)268-3890}}}%

\begin{document}
\maketitle
\vspace{4mm}
\begin{abstract}
Distributed consensus and other linear systems with system stochastic matrices $W_k$ emerge in various
settings, like opinion formation in social networks, rendezvous of
robots, and distributed inference in sensor networks. The matrices $W_k$ are often random, due to, e.g., random packet dropouts in wireless sensor networks. Key in analyzing
the performance of such systems is studying convergence of matrix products $W_kW_{k-1} \cdots W_1$.
 In this paper, we find the exact exponential rate $I$ for the convergence in probability
 of the product of such matrices when time $k$ grows large, under the assumption that the $W_k$'s are symmetric and independent identically distributed in time. Further, for commonly used random models like with gossip and
 link failure, we show that the rate $I$ is found by solving a min-cut problem and,
 hence, easily computable. Finally, we apply our results to optimally allocate the sensors' transmission power
 in consensus+innovations distributed detection.
\end{abstract}

\hspace{.43cm}\textbf{Keywords:} Consensus, consensus+innovations, performance analysis,
random network, convergence in probability, exponential rate.
\newpage

\section{Introduction}
\label{section-intro}
Linear systems with stochastic system matrices $W_k$ find applications in sensor~\cite{SoummyaConferenceConsensus},
multi-robot~\cite{Multi-Robot},
 and social networks~\cite{Golub-Jackson}. For example, in modeling opinion formation in social networks~\cite{Golub-Jackson},  individuals set their new opinion to the weighted average of their own opinion
 and the opinions of their neighbors.
 These systems appear both as autonomous, like consensus or gossip algorithms~\cite{BoydGossip}, and as input-driven algorithms, like consensus+innovations distributed inference~\cite{SoummyaEst}.
 Frequently, the system matrices $W_k$ are \emph{random}, like, for example, in consensus in wireless sensor networks,  due to either the use of a randomized
 protocol like gossip~\cite{BoydGossip}, or to link failures--random packet dropouts. In this paper, we determine the
      exact convergence rate of products of
      random, independent identically distributed (i.i.d.) general symmetric stochastic\footnote{By stochastic, we mean a nonnegative matrix whose rows sum to $1$. Doubly stochastic matrices besides row have also column sums equal to $1$.} matrices $W_k$, see Section~\ref{Sec-Main}. In particular, they apply to gossip and link failure.  For example, with gossip on a graph $G$, each realization of $W_k$ has the sparsity structure of the Laplacian matrix of a one link subgraph of $G$, with positive entries being arbitrary, but that we assume bounded \nolinebreak away \nolinebreak from \nolinebreak zero.

When studying the convergence of
products $W_k W_{k-1}...W_1$, it is well known that, when the modulus of the second largest eigenvalue of ${\mathbb E}\left[ W_k \right]$ is strictly less than
$1$, this product converges to $J:=\frac{1}{N}11^\top$ almost surely~\cite{JadbabaieErgodic} and, thus, in probability, i.e.,
for any $\epsilon>0$,
\begin{equation}
\label{eqn-prob-phi-j}
\mathbb P\left(\left\|W_k\cdots W_1-J\right\| \geq \epsilon \right) \rightarrow 0\,\,\mathrm{when\,\,}k\rightarrow \infty,
\end{equation}
where $\|\cdot\|$ denotes the spectral norm. This probability converges exponentially fast to zero with $k$~\cite{jadbabaie_on_consensus}, but, so far as we know, the exact convergence rate has not yet been computed. In this work, we compute the exact exponential rate of decay of the probability in~\eqref{eqn-prob-phi-j}.

\mypar{Contributions} Assuming that the non-zero entries of $W_k$ are bounded away from zero, we compute the exact exponential decay rate of the probability in~\eqref{eqn-prob-phi-j} by solving with equality (rather than lower and upper bounds) the corresponding large deviations limit, for every $\epsilon>0$:
\begin{equation}
\label{eq-consensus-rate}
\lim_{k \rightarrow \infty}\frac{1}{k}\log \mathbb P\left(\left\|W_k\cdots W_1-J\right\| \geq \epsilon \right)= - I,
\end{equation}
where the convergence rate~$I\geq 0$. Moreover, we characterize the rate $I$ and show that it does not depend on $\epsilon$. Our results reveal that the exact rate~$I$ is solely a function of the graphs induced by the matrices $W_k$ and the corresponding probabilities of occurrences of these graphs. In general, the computation of the rate $I$ is a combinatorial problem. However, for special important cases, we can get particularly simple expressions. For example, for a gossip on a connected tree, the rate is equal to $|\log(1-p_{ij})|$, where $p_{ij}$  is the probability of the link that is least likely to occur. Another example is with symmetric structures,
like uniform gossiping and link failures over a regular graph for which we show that the rate $I$ equals $|\log p_{\mathrm{isol}}|$, where $p_{\mathrm{isol}}$ is
 the probability that a node is isolated from the rest of the network.
For gossip with more general graph structures, we show that the rate $I=|\log(1-c)|$ where $c$ is the min-cut value
(or connectivity~\cite{compactingcuts}) of a graph whose links are weighted by the gossip link probabilities;
 the higher the connectivity $c$ is (the more costly or, equivalently, less likely it is to disconnect the graph) the larger
  the  rate $I$ and the faster the convergence are. Similarly, with link failures on general graphs, the rate is computed by solving a min-cut problem and is computable in polynomial time.

We now explain the intuition behind our result. To this end, consider the probability in~\eqref{eqn-prob-phi-j} when $\epsilon=1$\footnote{ It turns out, as we will show in Section~\ref{Sec-Main}, that the rate does not depend on $\epsilon$. Remark also that, because the matrices $W_k$ are stochastic, the spectral norm of $W_k\cdots W_1-J$ is less or equal to $1$ for all realizations of $W_1$,\ldots, $W_k$. Thus, the probability in~\ref{eqn-prob-phi-j} is equal to $0$ for $\epsilon>1$. }, i.e.,
when the norm of $\prod_{t=1}^k W_k-J$ stays equal to $1$. This happens only if the supergraph of all the graphs associated with the matrix realizations $W_1,\ldots,W_k$ is disconnected.
Motivated by this insight, we define the set of all possible graphs induced by the matrices $W_k$,
i.e., the set of realizable graphs,
and introduce the concept of disconnected collection of such graphs.
 For concreteness, we explain this here assuming gossip on a connected tree with $M$ links. For gossip on a connected tree, the set of realizable graphs consists of all one-edge subgraphs of the tree (and thus is of size $M$). If any fixed $j<M$ graphs were removed from this collection, the supergraph of the remaining graphs is disconnected; this collection of the remaining graphs is what we call a disconnected collection.
 Consider now the event that all the graph realizations (i.e., activated links) from time $t=1$ to time $t=k$ belong to a fixed disconnected collection, obtained, for example, by removal of one fixed one-edge graph.
 Because there were two isolated components in the network, the norm of $\prod_{t=1}^k W_k-J$ would under this event stay equal to $1$. The probability of this event is $M(1-p)^k$, where we assume that the links occur with the same probability $p=\frac{1}{M}$. Similarly, if all the graph realizations belong to a disconnected collection obtained by removal of $j$ one-edge graphs, for $1\leq j<M$, the norm remains at $1$, but now with probability ${M \choose j}(1-jp)^k$.
 For any event indexed by $j$ from this graph removal family of events, the norm stays at $1$ in the long run, but what will determine the rate is the most likely of all such events. In this case, the most likely event is that a single one-edge graph remains missing from time $1$ to time $k$, the probability of which is $M(1-p)^k$, yielding the value of the rate $I=|\log (1-p) |$.
This insight that the rate~$I$ is determined by the probability of the most likely disconnected collection of graphs  extends to the general matrix process.

\mypar{Review of the literature}
%
%
%
%
There has been a large amount
of work on linear systems driven by stochastic matrices. Early
work includes~\cite{tsitsiklisThesis84, DeGroot}, and the topic
received renewed interest in the past decade~\cite{jabaNeghbor,MurraySwitching}.
 Reference~\cite{MurraySwitching}
 analyzes convergence of the consensus algorithm
 under deterministic time-varying matrices $W_k$. Reference~\cite{BoydGossip}
  provides a detailed study of the standard gossip model,
  that has been further modified, e.g., in \cite{DimakisGossip,RabbatGossip};
   for a recent survey, see~\cite{dimakiskarmourarabbatscaglione-2010}. Reference
 \cite{JadbabaieErgodic} analyzes convergence under random matrices $W_k$,
 not necessarily symmetric, and ergodic -- hence not necessarily independent
 in time. Reference~\cite{Consensus-Delays} studies effects
 of delays, while reference~\cite{Consensus-Quantization} studies the impact of quantization. Reference~\cite{Yilin} considers random matrices $W_k$ and addresses the issue of
the communication complexity of consensus algorithms. The recent
 reference \cite{Consensus-Olshevsky-Tsitsiklis} surveys consensus and averaging algorithms and provides tight bounds on the worst case averaging times for deterministic
 time varying networks.
 In contrast with consensus (averaging) algorithms,
 \emph{consensus+innovations} algorithms include both a local averaging term (consensus) and an
 innovation term (measurement) in the state update process. These
 algorithms find applications in distributed
 inference in sensor networks,
 see, e.g.,~\cite{SoummyaEst,Sayed-LMS,Stankovic-Estimation} for distributed estimation,
 and, e.g.,~\cite{running-consensus-detection,GaussianDD,Non-Gaussian-DD}, for distributed detection.
  In this paper, we illustrate the usefulness of the rate of consensus $I$ in the context of
   a consensus+innovations algorithms by optimally allocating
  the transmission power of sensors for distributed detection.

Products of random matrices appear also in many other fields that use techniques drawn from Markov process theory. Examples include repeated interaction dynamics in quantum systems~\cite{BruneauJoyeMerkli}, inhomogeneous Markov chains with random transition matrices~\cite{TouriNedic,BruneauJoyeMerkli}, infinite horizon control strategies for Markov chains and non-autonomous linear differential equations~\cite{Leizarowitz}, or discrete linear inclusions~\cite{TouriNedic-backward-products}. These papers are usually concerned with deriving convergence results on these products and determining the limiting matrix. Reference~\cite{BruneauJoyeMerkli} studies the product of matrices belonging to a class of complex contraction matrices and characterizes the limiting matrix by expressing the product as a sum of a decaying process, which exponentially converges to zero, and a fluctuating process. Reference~\cite{Leizarowitz} establishes conditions for strong and weak ergodicity for both forward and backward products of stochastic matrices, in terms of the limiting points of the matrix sequence. Using the concept of infinite flow graph, which the authors introduced in previous work, reference~\cite{TouriNedic} characterizes the limiting matrix for the product of stochastic matrices in terms of the topology of the infinite flow graph.
For more structured matrices,~\cite{DiakonisWood} studies products of nonnegative matrices. For nonnegative matrices, a comprehensive study of the asymptotic behavior of the products can be found in~\cite{SenetaBook}. A different line of research, closer to our work, is concerned with the limiting distributions of the products (in the sense of the central limit theorem and large deviations). The classes of matrices studied are: invertible matrices~\cite{Tutubalin,Gauivarc'hRaugi} and its subclass of matrices of determinant equal to~$1$~\cite{LePage} and, also, positive matrices~\cite{Hennion}. None of these apply to our case, as the matrices that we consider might not be invertible ($W_k-J$ has a zero eigenvalue, for every realization of $W_k$) and, also, we allow the entries of $W_k$ to be zero, and therefore the entries of $W_k-J$ might be negative with positive probability.
Furthermore, as pointed out in~\cite{Kargin}, the results obtained in~\cite{Tutubalin,Gauivarc'hRaugi,LePage}  do not provide ways to effectively compute the rates of convergence. Reference~\cite{Kargin} improves on the existing literature in that sense by deriving more explicit \textit{bounds} on the convergence rates, while showing that, under certain assumptions on the matrices, the convergence rates do not depend on the size of the matrices; the result is relevant from the perspective of large scale dynamical systems, as it shows that, in some sense, more complex systems are not slower than systems of smaller scale, but again it does not apply to our study.

To our best knowledge, the \textit{exact} large deviations rate $I$ in~\eqref{eq-consensus-rate} has
not been computed for i.i.d. averaging matrices $W_k$, nor for the commonly used
sub-classes of gossip and link failure models. Results in the existing
literature provide upper and lower bounds on the rate $I$, but not the
exact rate $I$. These bounds are based on the second
largest eigenvalue of $\mathbb{E}[W_k]$ or $\mathbb{E}[W^2_k]$, e.g.,~\cite{BoydGossip,weight-opt,JadbabaieErgodic}.
 Our result~\eqref{eq-consensus-rate} refines these existing bounds, and sheds
 more light on the asymptotic convergence of the probabilities in~\eqref{eqn-prob-phi-j}.
 For example, for the case when each realization of $W_k$ has a connected underlying support graph (the case studied in~\cite{MurraySwitching}), we calculate the rate $I$ to
  be equal $+\infty$ (see Section~\ref{Sec-Main}), i.e., the convergence of the probabilities in~\eqref{eqn-prob-phi-j}
  is faster than exponential. On the other hand, the ``rate'' that would
  result from the bound based on $\lambda_2(\mathbb{E}[W^2_k])$ is finite unless
  $W_k \equiv J$. This is particularly relevant
  with consensus+innovations algorithms, where, e.g.,
  the consensus+innovations distributed detector is asymptotically
  optimal if $I=\infty$,~\cite{allerton}; this fact cannot be seen from the bounds
  based on $\lambda_2(\mathbb{E}[W^2_k])$.

The rate~$I$ is a valuable metric for the design of algorithms (or linear systems) driven by system
 matrices $W_k$, as it determines the algorithm's asymptotic performance and is easily computable
 for commonly used models. We demonstrate the usefulness
 of~$I$ by optimizing the allocation of the sensors' transmission power in a sensor
 network with fading (failing) links, for the purpose of
 distributed detection with the consensus+innovations algorithm~\cite{GaussianDD,Non-Gaussian-DD}.

\mypar{Paper organization} Section \ref{section-setup}
introduces the model for random matrices $W_k$ and
defines relevant quantities needed in the sequel.
Section~\ref{Sec-Main} proves the result
on the exact exponential rate~$I$ of consensus.
Section \ref{sec-Examples} shows
 how to compute the rate $I$ for gossip and link
 failure models via a min-cut problem.
  Section \ref{section-application}
  addresses  optimal power allocation
  for distributed detection by maximizing the rate $I$.
   Finally, section \ref{section-conclusion} concludes the paper.

\section{Problem setup}
\label{section-setup}

\mypar{Model for the random matrices $W_k$}
Let $\left\{W_k:\,k=1,2,...\right\}$ be a discrete time (random) process where $W_k$, for all $k$,
 takes values in the set of doubly stochastic, symmetric, $N \times N$ matrices.
 \begin{assumption}
 \label{assumption}
 We assume the following.
 \begin{enumerate}
 \item  The random matrices $W_k$ are independent identically distributed (i.i.d.).
 \item The entries of any realization $W$ of $W_k$ are bounded away from $0$ whenever positive. That is,
there exists a scalar $\delta$, such that, for any realization $W$, if $W_{ij}>0$, then $W_{ij}\geq \delta$. An entry of $W_k$ with positive value, will be called an active entry.
\item For any realization $W$, for all $i$, $W_{ii}\geq \delta$.
\end{enumerate}
\end{assumption}
Also, let $\mathcal W$ denote the set of all possible realizations of $W_k$.

\mypar{Graph process}
For a doubly stochastic symmetric matrix $W$, let $G(W)$ denote its induced undirected graph, i.e., $G(W)=\left(V,E(W)\right)$, where $V=\{1,2,\ldots,N\}$ is the set of all nodes and \[E(W)=\left\{ \{i,j\}\in {V\choose 2}: W_{ij}>0 \right\} .\] We define the random graph process $\left\{G_t: t=1,2,\ldots\right\}$
through the random matrix process $\{W_k:\,k=1,2,...\}$ by: $G_t=G(W_t)$, for $t=1,2,\ldots$. As the matrix process is i.i.d., the graph process is i.i.d. as well. We collect the underlying graphs of all possible matrix realizations $W$ (in $\mathcal W$) in the set $\mathcal G$:
\begin{equation}
\mathcal G:=\left\{ G(W): W \in \mathcal W\right\}.
\end{equation}
Thus, the random graphs $G_t$ take their realizations from $\mathcal G$. Similarly, as with the matrix entries, if $\{i,j\}\in E(G_t)$, we call $\{i,j\}$ an active link.

We remark that the conditions on the random matrix process from Assumption~\ref{assumption} are satisfied automatically for any i.i.d. model with finite space of matrices $\mathcal W$ ($\delta$ could be taken to be the minimum over all positive entries over all matrices from $\mathcal W$). We illustrate with three instances of the random matrix model the case when the (positive) entries of matrix realizations can continuously vary in certain intervals, namely, gossip, $d$-adjacent edges at a time, and link failures.

\begin{example}[Gossip model] Let $G=(V,E)$ be an arbitrary connected graph on $N$ vertices. With the gossip algorithm on the graph $G$, every realization of $W_k$ has exactly two off diagonal entries that are active: $[W_k]_{ij}=[W_k]_{ji}>0$, for some $\{i,j\}\in G$, where the entries are equal due to the symmetry of $W_k$. Because $W_k$ is stochastic, we have that $[W_k]_{ii}=[W_k]_{jj}=1-[W_k]_{ij}$, which, together with Assumption~\ref{assumption}, implies that $[W_k]_{ij}$ must be bounded (almost surely) by $\delta \leq [W_k]_{ij}\leq 1-\delta$. Therefore, the set of matrix realizations in the gossip model is:
\begin{align*}
\mathcal W^{\mbox{ {\scriptsize Gossip}}}=\bigcup_{\{i,j\} \in E}
& \left\{ A\in \mathbb R^{N\times N}: A_{ij}=A_{ji}=\alpha ,\: A_{ii}=A_{jj}=1-\alpha,\: \alpha \in [\delta, 1-\delta],\right. \\[-7pt]
& \left. \phantom{A\in\mathbb R^{N\times N}:}\;\;\; A_{ll}=1, \: \mathrm{for\,}l\neq i,j,\: A_{ml}=0, \mathrm{for\,} l\neq m \,\mathrm{and\,} l,m\neq i,j\right\}.
\end{align*}
\end{example}

\begin{example}[Averaging model with d-adjacent edges at a time] Let $G_d=(V,E)$ be a $d$-regular connected graph on $N$ vertices, $d\leq N-1$.
Consider the following averaging scheme where exactly $2d$ off-diagonal entries of $W_k$ are active at a time: $[W_k]_{ij}=[W_k]_{ji}>0$, for some fixed $i\in V$ and all $j\in V$ such that $\{i,j\}\in E$. In other words, at each time in this scheme, the set of active edges is the set of edges adjacent to some node $i\in V$. Taking into account Assumption~\ref{assumption} on $W_k$, the set of matrix realizations for this averaging model is:
\begin{align*}
\mathcal W^{\mathrm{ {\scriptsize d-adjacent}}}=& \bigcup_{i \in V}
 \left\{A\in \mathbb R^{N\times N}: A=A^\top, A_i=v, v\in \mathbb R^N,  v_j=0, \mathrm{if\,} \{i,j\}\notin E, 1^\top v=1,
 v\geq \delta, \right.\\[-6pt]
 & \left.  \phantom{ 1^\top v=1}\quad \quad\quad\;\;   A_{jj}=1-A_{ij}, \mathrm{for\,} \{i,j\}\in E, A_{ll}=1\, \mathrm{and\,} A_{il}=0,\, \mathrm{for\,} \{i,l\}\notin E\right\},
\end{align*}
where $A_i$ denotes the $i$th column of matrix $A$.

\end{example}

\begin{example}[Link failure (Bernoulli) model] Let $G=(V,E)$ be an arbitrary connected graph on $N$ vertices. With link failures, occurrence of each edge in $E$ is a Bernoulli random variable and occurrences of edges are independent. Due to independence, each subgraph $H=(V,F)$ of $G$, $F\subseteq E$, is a realizable graph in this model. Also, for any given subgraph $H$ of $G$, any matrix $W$ with the sparsity pattern of the Laplacian matrix of $H$ and satisfying Assumption~\ref{assumption} is a realizable matrix. Therefore, the set of all realizable matrices in the link failure model is
\begin{align*}
\mathcal W^{\mbox{ {\scriptsize Link\,fail.}}}=\bigcup_{F \subseteq E}
 \left\{ A\in \mathbb R^{N\times N}: A=A^\top,  A_{ij}\geq \delta, \mathrm{if\,} \{i,j\}\in F,\: A_{ij}=0, \mathrm{if\,} \{i,j\}\notin F, \:A 1=1\right\}.
\end{align*}
\end{example}

\mypar{Supergraph of a collection of graphs and supergraph disconnected collections}
For a collection of graphs $\mathcal H$ on the same set of vertices $V$, let $\Gamma(\mathcal H)$
denote the graph that contains all edges from all graphs in $\mathcal H$. That is, $\Gamma(\mathcal H)$ is the minimal graph (i.e., the graph with the minimal number of edges) that is a supergraph of every graph in $\mathcal H$:
\begin{equation}
\Gamma(\mathcal H):=  (V, \bigcup_{G\in \mathcal H}E(G)),
\end{equation}
where $E(G)$ denotes the set of edges of graph $G$.

Specifically, we denote by $\Gamma(s,t)$\footnote{Graph $\Gamma(s,t)$ is associated with the matrix product $W_{s}\cdots W_{t+1}$ going from time $t+1$ until time $s>t$. The notation $\Gamma(s,t)$ indicates that the product is backwards; see also the definition of the product matrix $\Phi(s,t)$ in Section~\ref{Sec-Main}.}
 the random graph that collects the edges from
all the graphs $G_r$ that appeared from time $r=t+1$ to $r=s$, $s>t$, i.e.,
\[\Gamma(s,t):=\Gamma(\left\{{G_{s},G_{s-1},\ldots, G_{t+1}}\right\}).\]
Also, for a collection $\mathcal H \subseteq \mathcal G$ we use $p_{\mathcal H}$ to denote the probability that a graph realization $G_t$ belongs to \nolinebreak $\mathcal H$:
\begin{equation}
\label{eq-def-pH}
p_{\mathcal H}=\sum_{H\in \mathcal {H}} \mathbb P(G_t=H).
\end{equation}

We next define collections of realizable graphs of certain types that will be important in computing the rate in~\eqref{eq-consensus-rate}.

\begin{definition} The collection $\mathcal H \subset \mathcal G$ is a disconnected collection
of $\mathcal G$ if its supergraph $\Gamma(\mathcal H)$ is disconnected.
\end{definition}
Thus, a disconnected collection is any collection of realizable graphs such that the union of all of its graphs yields a disconnected graph. We also define the set of all possible disconnected collections of $\mathcal G$:
\begin{equation}
\label{eq-def-all-maximal-collections}
\Pi(\mathcal G)=\left\{\mathcal H\subseteq \mathcal G: \mathcal H \mathrm{\;is\; a\; disconnected\; collection\; on\;\mathcal G} \right\}.
\end{equation}
We further refine this set to find the largest possible disconnected collections on $\mathcal G$.

\begin{definition} We say that a collection $\mathcal H \subset \mathcal G$ is a maximal disconnected collection
of $\mathcal G$ (or, shortly, maximal) if:
\renewcommand{\labelenumi}{\roman{enumi}}
\begin{enumerate}[i)]
\item $\mathcal H \in \Pi(\mathcal G)$, i.e., $\mathcal H$ is a disconnected collection on $\mathcal G$; and
\item for every $G\in {\mathcal G}\setminus {\mathcal H}$,  $\Gamma(\mathcal H \cup G)$ is connected.
\end{enumerate}
\end{definition}
In words, $\mathcal H$ is maximal if the graph $\Gamma(\mathcal H)$ that collects all
edges of all graphs in $\mathcal H$ is disconnected, but,
adding all the edges of any of the remaining graphs (that are not in $\mathcal H$) yields a connected graph.
We also define the set of all possible maximal collections of $\mathcal G$:
\begin{equation}
\label{eq-def-all-maximal-collections}
\Pi^\star(\mathcal G)=\left\{\mathcal H\subseteq \mathcal G: \mathcal H \mathrm{\;is\; a\; maximal\; collection\; on\;\mathcal G} \right\}.
\end{equation}
We remark that $\Pi^\star(\mathcal G)\subseteq \Pi(\mathcal G)$.
We now illustrate the set of all possible graph realizations $\mathcal G$, and
its maximal collections $\mathcal H$ with two examples.

\begin{example}[Gossip model] If the random matrix process is defined by the gossip algorithm on the full graph on $N$ vertices, then $\mathcal G=\left\{(V, \{i,j\}): \{i,j\} \in {V \choose 2} \right\}$; in words, $\mathcal G$ is the set of all possible one-link graphs on $N$ vertices. An example of a maximal collection of $\mathcal G$ is \[\mathcal G\setminus \left\{(V, \{i,j\}): j=1,\ldots N, j\neq i \right\},\] where $i$ is a fixed vertex, or, in words, the collection of all one-link graphs except of those whose link is adjacent to $i$. Another example is \[\mathcal G\setminus \left(\left\{(V, \{i,k\}): k=1,\ldots N, k\neq i, k\neq j \right\}\cup \left\{(V, \{j,l\}): l=1,\ldots N, l\neq i, l\neq j \right\}\right).\]
\end{example}
\begin{example}[Toy example] Consider a network of five nodes with the set of realizable graphs $\mathcal G=\{G_1,G_2,G_3\}$, where the graphs $G_i$, $i=1,2,3$ are given in Figure~\ref{Fig-Example-Toy}. In this model, each realizable graph is a two-link graph, and the supergraph of all the realizable graphs $\Gamma(\{G_1,G_2,G_3\})$ is connected.
\begin{figure}[thpb]
  \centering
  \includegraphics[scale=0.3, trim =  0 0 0 110mm ]{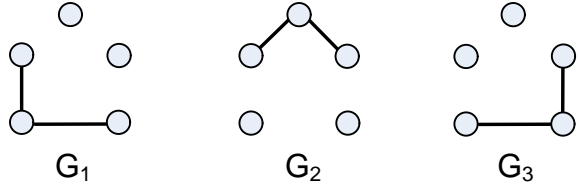}
  \caption{Example of a five node network with three possible graph realizations, each being a two-link graph}
\label{Fig-Example-Toy}
\end{figure}
\vspace{-1mm}
If we scan over the supergraphs $\Gamma(\mathcal H)$ of all subsets $\mathcal H$ of $\mathcal G$, we see that $\Gamma(\{G_1,G_2\})$, $\Gamma(\{G_2,G_3\})$ and $\Gamma(\{G_1,G_2,G_3\})$ are connected, whereas the $\Gamma(\{G_1,G_3\})$, and $\Gamma(G_i)=G_i$, $i=1,2,3$ are disconnected. Therefore, $\Pi(\mathcal G)=\{\{G_1\}, \{G_2\}, \{G_3\}, \{G_1,G_3\} \}$ and $\Pi^\star(\mathcal G)=\{\{G_2\}, \{G_1,G_3\} \}$.

\end{example}

We now observe that, if the graph $\Gamma(s,t)$ that collects all
the edges that appeared from time $t+1$ to time $s$ is disconnected, then
all the graphs $G_r$ that appeared from $r=t+1$ through $r=s$ belong to some
maximal collection $\mathcal H$.
\begin{observation}
\label{obs-there-exists-maximal-for-any-disc} If for some $s$ and $t$, $s>t$, $\Gamma(s,t)$ is disconnected, then there exists
a maximal collection $\mathcal H \in \Pi^\star(\mathcal G)$, such that $G_r \in \mathcal H$, for every $r$, $t<r\leq s$.
\end{observation}
\section{Exponential rate for consensus}
\label{Sec-Main}
Denote $\Phi(s,t) := W_s W_{s-1}\cdots W_{t+1}$, and $\widetilde \Phi(s,t) := \Phi(s,t) -J$, for $s>t\geq 0$.
The following Theorem gives the exponential decay rate of the probability
$\mathbb P\left( \left\|\widetilde \Phi(k,0)\right\| \geq \epsilon \right)$.
\begin{theorem}
\label{theorem-main}
Consider the random process $\left\{ W_k:\,k=1,2,\ldots \right\}$ under Assumption~\ref{assumption}. Then:
\[\lim_{k \rightarrow \infty} \frac{1}{k} \log \mathbb P\left( \left\|\widetilde \Phi(k,0)\right\| \geq \epsilon \right)=-I,
\:\: \forall \epsilon \in (0,1]\]
where
\vspace{-1mm}
\begin{equation*}
I=\left\{  \begin{array}{ll}   +\infty & \mathrm{if}\; \Pi^\star(\mathcal G)=\emptyset\\   |\log p_{\max}| & \mathrm{otherwise} \\  \end{array}\right.,
\end{equation*}
and
\[p_{\max}=\max_{\mathcal H \in \Pi^\star(\mathcal G)} p_{\mathcal H}\]
is the probability of the most likely maximal disconnected collection.
\end{theorem}
To prove Theorem \ref{theorem-main}, we first consider the case when $\Pi^\star(\mathcal G)$ is nonempty, and thus when $p_{\max}>0$. In this case, we find the rate $I$ by showing the lower and the upper bounds:
\begin{eqnarray}
\label{eqn-lower-bound}
\liminf_{k \rightarrow \infty} \frac{1}{k} \log \mathbb P\left( \left\|\widetilde \Phi(k,0)\right\| \geq \epsilon \right)
&\geq& \log p_{\mathrm{max}}\\
\label{eqn-upper-bound}
\limsup_{k \rightarrow \infty} \frac{1}{k} \log \mathbb P\left( \left\|\widetilde \Phi(k,0)\right\| \geq \epsilon \right)
&\leq& \log p_{\mathrm{max}}.
\end{eqnarray}
Subsection \ref{sub-Sec-Lowerbound} proves the lower bound \eqref{eqn-lower-bound},
and subsection \ref{sub-Sec-upper-bound} proves the upper bound \eqref{eqn-upper-bound}.
%
%
%

\vspace{-1.5mm}
\subsection{Proof of the lower bound \eqref{eqn-lower-bound}}
\label{sub-Sec-Lowerbound}
We first find the rate for the probability that
the network stays disconnected over the interval $1,...,k$.
\begin{lemma}
\label{lemma-disconnected}
\vspace{-2mm}
\[\lim_{k \rightarrow \infty} \frac{1}{k} \log \mathbb P\left( \Gamma(k,0)\; \mathrm{is\; disconnected}\, \right)=\log p_{\mathrm{max}}.\]
%
\end{lemma}
Having Lemma \ref{lemma-disconnected}, the lower bound \eqref{eqn-lower-bound} follows from the following relation:
\begin{eqnarray*}
\label{eq-lower-bound-first}
\mathbb P\left(\left\|\widetilde \Phi(k,0)\right\|\geq \epsilon\right)&\geq&  \mathbb P\left(\left\|\widetilde \Phi(k,0)\right\|= 1\right)\\
&=& \mathbb P\left( \Gamma(k,0)\; \mathrm{is\; disconnected}\, \right),
\end{eqnarray*}
that is stated and proven in Lemma~\ref{lemma-the-main-equivalence} further ahead.

\begin{proof}[Proof of Lemma \ref{lemma-disconnected}]If all the graph realizations until time $k$ belong to a certain maximal collection $\mathcal H$, by definition of a maximal collection, $\Gamma(k,0)$ is disconnected with probability $1$. Therefore, for any maximal collection $\mathcal H$, the following bound holds:
\begin{eqnarray}
\label{eq-lower-bound-second}
 \mathbb P\left( \Gamma(k,0)\; \mathrm{is\; disconnected}\, \right)\geq  \mathbb P\left( G_t \in \mathcal H, \;\forall\,t=1,\ldots,k  \right) = p_{\mathcal H}^k. \nonumber
\end{eqnarray}
The best bound, over all maximal collections $\mathcal H$, is the one that corresponds to the ``most likely'' maximal collection:
\begin{equation}
\label{eq-lowerB-for-disconnected}
 \mathbb P\left( \Gamma(k,0)\; \mathrm{is\; disconnected}\, \right) \geq p_{\max}^k.
\end{equation}

We will next show that an upper bound with the same rate of decay (equal to $p_{\max}$) holds for the probability of the network staying disconnected. To show this, we reason as follows: if $\Gamma(k,0)$ is disconnected, then all the graph realizations until time $k$, $G_1,\ldots,G_k$, belong to some maximal collection.
It follows that
\begin{eqnarray*}
\mathbb P\left(\Gamma(k,0)\;\mathrm{is \;disconnected\,}\right)&=&
\mathbb P\left(\bigcup_{\mathcal H \in \Pi^\star(\mathcal G)}\left\{ G_t\in \mathcal H,\;\mathrm{for}\; t=1,\ldots,k\right\}\right)\\
&\leq&\sum_{\mathcal H \in \Pi^\star(\mathcal G)} \mathbb P\left( G_t\in \mathcal H,\;\mathrm{for}\; t=1,\ldots,k\right)\\
&=& \sum_{\mathcal H \in \Pi^\star(\mathcal G)} p_{\mathcal H}^k.
\end{eqnarray*}
Finally, we bound each term in the previous sum by the probability $p_{\max}$ of the most likely maximal collection, and we obtain:
\begin{equation}
\label{eq-upperB-for-disconnected}
\mathbb P\left(\Gamma(k,0)\;\mathrm{is \;disconnected}\right) \leq \left|\Pi^\star(\mathcal G) \right|  p_{\max}^k,
\end{equation}
where $\left|\Pi^\star(\mathcal G) \right|$ is the number of maximal collections on $\mathcal G$.

Combining~\eqref{eq-lowerB-for-disconnected} and~\eqref{eq-upperB-for-disconnected} we get:
\begin{equation*}
p_{\max}^k \leq \mathbb P\left(\Gamma(k,0)\;\mathrm{is \;disconnected}\right) \leq \left|\Pi^\star(\mathcal G) \right|  p_{\max}^k.
\end{equation*}
which implies
\begin{equation*}
\lim_{k\rightarrow \infty} \frac{1}{k}\log \mathbb P\left(\Gamma(k,0)\;\mathrm{is \;disconnected}\right)=\log p_{\max}.
\end{equation*}
\vskip -5mm
\end{proof}
\subsection{Proof for the upper bound in \eqref{eqn-upper-bound}}
\label{sub-Sec-upper-bound}
The next lemma relates the products of the weight matrices $\Phi(s,t)$ with the corresponding graph $\Gamma(s,t)$ and is the key point in our analysis.
Recall that
$\widetilde{\Phi}(s,t):=\Phi(s,t)-J.$
\begin{lemma}
\label{lemma-delta-bounds} For any realization of matrices $W_r\in \mathcal W$, $r=t+1,\ldots,s$, $s>t$:\footnote{The statements of the results in subsequent Corollary~\ref{corollary-just-connected} and Lemma~\ref{lemma-the-main-equivalence} are also in the point-wise sense.}
\begin{enumerate}
\item \label{lemma-delta-bounds-part-a}if $[\Phi(s,t)]_{ij}>0$, $i\neq j$, then $[\Phi(s,t)]_{ij}\geq\delta^{s-t}$;
\item \label{lemma-delta-bounds-part-b} for $i=1,\ldots,N$ $[\Phi(s,t)]_{ii}\geq \delta^{s-t}$;
\item \label{eq-delta-bound-entries} if $[\Phi(s,t)]_{ij}>0$, $i\neq j$, then  $\left[\Phi(s,t)^\top\Phi(s,t)\right]_{ij}\geq \delta^{2(s-t)};$
\item \label{eq-delta-bound-spectral-radius}
$\|\widetilde \Phi(s,t)\|\leq \left(1-\delta^{2(s-t)}\lambda_{\mathrm F}\left(L \left(\Gamma(s,t)\right)\right)\right)^{\frac{1}{2}},$
\end{enumerate}
where $L(G)$ is the Laplacian matrix of the graph $G$, and
$\lambda_{\mathrm{F}}(A)$ is the second smallest eigenvalue (the Fiedler eigenvalue) of a positive semidefinite matrix $A$.
\end{lemma}
\begin{proof} Parts~\ref{lemma-delta-bounds-part-a} and~\ref{lemma-delta-bounds-part-b} are a consequence of the fact that the positive entries of the weight matrices are bounded below by $\delta$ by Assumption~\ref{assumption}; for the proofs of~\ref{lemma-delta-bounds-part-a} and~\ref{lemma-delta-bounds-part-b}, see~\cite{Angelia}, Lemma~1 a), b). Part~\ref{eq-delta-bound-entries} follows from parts~\ref{lemma-delta-bounds-part-a} and~\ref{lemma-delta-bounds-part-b}, by noticing that, for all $\{i,j\}$, $i\neq j$, such that $[\Phi(s,t)]_{ij}>0$, we have:
\begin{eqnarray*}
\left[\Phi(s,t)^\top\Phi(s,t)\right]_{ij}&=&\sum_{l=1}^N [\Phi(s,t)]_{li} [\Phi(s,t)]_{lj}\\
 &\geq& [\Phi(s,t)]_{ii} [\Phi(s,t)]_{ij}\\
&\geq& \left(\delta^{s-t}\right)^2.
\end{eqnarray*}

To show part~\ref{eq-delta-bound-spectral-radius}, we notice first that $\left\|\widetilde \Phi(s,t)\right\|^2$ is the second largest eigenvalue of $\Phi(s,t)^\top \Phi(s,t)$, and, thus, can be computed as
\vskip -3mm
\begin{eqnarray*}
\left\|\widetilde \Phi(s,t)\right\|^2 &=& \max_{q^\top q=1,\, q \perp 1 } q^\top \Phi(s,t)^\top \Phi(s,t) q\\
\end{eqnarray*}
\vskip -7mm
Since $\Phi(s,t)^\top \Phi(s,t)$ is a symmetric stochastic matrix, it can be shown, e.g.,~\cite{MurraySwitching}, that its quadratic form can be written as:
\begin{eqnarray}
\label{eq-quadratic-form-of-Phi}
q^\top \Phi(s,t)^\top \Phi(s,t) q &=& q^\top q - \sum_{\{i,j\}} \left[\Phi(s,t)^\top \Phi(s,t)\right]_{ij}\left(q_i-q_j\right)^2\nonumber \\
&\leq& 1 -  \delta^{2(s-t)} \sum_{\{i,j\}: \left[\Phi(s,t)^\top \Phi(s,t)\right]_{ij}>0} \left(q_i-q_j\right)^2.
\end{eqnarray}
where the last inequality follows from part~\ref{eq-delta-bound-entries}. Further, if the graph $\Gamma(s,t)$ contains some link $\{i,j\}$, then, at some time $r$, $t<r\leq s$, a realization $W_r$ with $[W_r]_{ij}>0$ occurs. Since the diagonal entries of all the realizations of the weight matrices are positive (and, in particular, those from time $r$ to time $s$), the fact that $[W_r]_{ij}>0$ implies that $[\Phi(s,t)]_{ij}>0$. This, in turn, implies
\begin{equation*}
\left\{\{i,j\} \in \Gamma(s,t)\right\}\subseteq \left\{\{i,j\}: \left[\Phi(s,t)^\top \Phi(s,t)\right]_{ij}>0 \right\}.
\end{equation*}
Using the latter, and the fact that the entries of $\Phi(s,t)^\top \Phi(s,t)$ are non-negative, we can bound the sum in~\eqref{eq-quadratic-form-of-Phi} over $\left\{\{i,j\}: \left[\Phi(s,t)^\top \Phi(s,t)\right]_{ij}>0 \right\}$ by the sum over $\left\{\{i,j\} \in \Gamma(s,t)\right\}$ only, yielding
\begin{equation*}
q^\top \Phi(s,t)^\top \Phi(s,t) q \leq 1 - \delta^{2(s-t)} \sum_{\{i,j\} \in \Gamma(s,t)}\left(q_i-q_j\right)^2.
\end{equation*}
Finally, $\min_{q^\top q=1,\, q \perp 1 }\sum_{\{i,j\} \in \Gamma(s,t)}\left(q_i-q_j\right)^2$ is equal to the Fiedler eigenvalue (i.e., the second smallest eigenvalue) of the Laplacian $L(\Gamma(s,t))$. This completes the proof of part~\ref{eq-delta-bound-spectral-radius} and Lemma~\ref{lemma-delta-bounds}.
\end{proof}
We have the following corollary of part~\ref{eq-delta-bound-spectral-radius} of Lemma~\ref{lemma-delta-bounds}, which, for a fixed interval length $s-t$, and for the case when $\Gamma(s,t)$ is connected, gives a uniform bound for the spectral norm of $\widetilde\Phi(s,t)$.
\begin{corollary}
\label{corollary-just-connected}
For any $s$ and $t$, $s>t$, if $\Gamma(s,t)$ is connected, then
\begin{equation}
\label{eq-uniform-bound-for-connected-s-t-intervals}
\left\|\widetilde \Phi(s,t)\right\| \leq
\left(1-c \delta^{2(s-t)}\right)^{\frac{1}{2}},
\end{equation}
where $c=2(1-\cos\frac{\pi}{N})$ is the Fiedler value of the path graph on $N$ vertices, i.e.,
the minimum of $\lambda_{\mathrm F}(L(G))>0$
over all connected graphs on $N$ vertices~\cite{min-connectivity} .
\end{corollary}
\begin{proof}
The claim follows from part~\ref{eq-delta-bound-spectral-radius} of Lemma~\ref{lemma-delta-bounds} and from the fact that for connected $\Gamma(s,t)$: $c=\min_{G\;\mathrm{is\;connected}} \lambda_{\mathrm F}(L(G))\leq \lambda_{\mathrm F}(L(\Gamma(s,t)))$. 
\end{proof}

The previous result, as well as part~\ref{eq-delta-bound-spectral-radius} of Lemma~\ref{lemma-delta-bounds}, imply that, if the graph $\Gamma(s,t)$ is connected, then the spectral norm of $\widetilde \Phi(s,t)$ is smaller than $1$. It turns out that the connectedness of $\Gamma(s,t)$ is not only sufficient, but it is also a necessary condition for $\left\|\widetilde \Phi(s,t)\right\|<1$. Lemma~\ref{lemma-the-main-equivalence} explains this.
\begin{lemma}
\label{lemma-the-main-equivalence}
For any $s$ and $t$, $s>t$:
\[\left\|\widetilde\Phi(s,t)\right\|<1\; \Leftrightarrow\; \Gamma(s,t)\; \mathrm{is\;connected.}\]
\end{lemma}
\begin{proof}
We first show the \emph{if} part. Suppose $\Gamma(s,t)$ is connected. Then, $\lambda_{\mathrm F}\left(L \left(\Gamma(s,t)\right)\right)>0$ and the claim follows by part~\ref{eq-delta-bound-spectral-radius} of Lemma~\ref{lemma-delta-bounds}. We prove the \emph{only if} part by proving the following equivalent statement:
\[  \Gamma(s,t) \mathrm{\;is\;not\;connected} \Rightarrow  \left\|\widetilde\Phi(s,t)\right\|=1. \]
To this end, suppose that $\Gamma(s,t)$ is not connected and, without loss of generality, suppose that $\Gamma(s,t)$ has two components $C_1$ and $C_2$.  Then, for $i\in C_1$ and $j\in C_2$, $\{i,j\} \notin \Gamma(s,t)$, and, consequently, $\{i,j\} \notin G_r$,  for all $r$, $t<r\leq s$. By definition of $G_r$, this implies that the $i,j$-th entry in the corresponding weight matrix is equal to zero, i.e.,
\[  \forall r,\, t<r\leq s: [W_r]_{ij}=0, \forall \{i,j\} \mathrm{\;s.t.\;} i\in C_1,\,j\in C_2.  \]
Thus, every matrix realization $W_r$ from time $r=t+1$ to time $r=s$ has a block diagonal form (up to a symmetric permutation of rows and columns)  \[W_r=\left[ \begin{array}{cc}
[W_r]_{C_1} & 0  \\
0 & [W_r]_{C_2}  \end{array} \right],\] where $[W_r]_{C_1}$ is the block of $W_r$ corresponding to the nodes in $C_1$, and similarly for $[W_r]_{C_2}$. This implies that $\Phi(s,t)$ will have the same block diagonal form, which, in turn, proves that
$\left\|\widetilde \Phi(s,t)\right\|=1$. This completes the proof of the only if part and the proof of Lemma~\ref{lemma-the-main-equivalence}.
\end{proof}
%
%
%
We next define the sequence of stopping times $T_i$, $i=1,2,\ldots$ by:
\begin{align}
T_i&=\min\{t\geq T_{i-1}+1:\; \Gamma(t,T_{i-1}) \;\mathrm{is\;connected} \},\; \mathrm{for}\;i\geq 1,\\
T_0&=0. \nonumber
\end{align}
The sequence $\{T_i\}_{i\geq 1}$ defines the times when the network becomes connected, and, equivalently, when the averaging process makes an improvement (i.e., when the spectral radius of $\widetilde \Phi$ drops below $1$).

For fixed time $k\geq 1$, let $M_k$ denote the number of improvements until time $k$:
\begin{equation}
M_k=\max\left\{i\geq 0: T_i\leq k\right\}.
\end{equation}
We now explain how, at any given time $k$, we can use the knowledge of $M_k$ to bound the norm of the ``error'' matrix $\widetilde \Phi(k,0)$. Suppose that $M_k=m$. If we knew the locations of all the improvements until time $k$, $T_i=t_i$, $i=1,\ldots,m$ then, using eq.~\eqref{eq-uniform-bound-for-connected-s-t-intervals}, we could bound the norm of $\widetilde \Phi(k,0)$. Intuitively, since for fixed $k$ and fixed $m$ the number of allocations of $T_i$'s is finite, there will exist the one which yields the worst bound on $\left\|\widetilde \Phi(k,0)\right\|$. It turns out that the worst case allocation is the one with equidistant improvements, thus allowing for deriving a bound on $\left\|\widetilde \Phi(k,0)\right\|$ only in terms of $M_k$. This result is given in Lemma~\ref{lemma-Mk=m-bound}.
\begin{lemma}
\label{lemma-Mk=m-bound} For any realization of $W_1,W_2,\ldots,W_k$ and $k=1,2,\ldots$ the following holds:
\begin{equation}
\label{eq-lemma-condition-on-Mk=m}
\left\|\widetilde \Phi(k,0)\right\|\leq \left(1-c\delta^{2\frac{k}{M_k}}\right)^\frac{M_k}{2}.
\end{equation}
\end{lemma}

\begin{proof} Suppose $M_k=m$ and $T_1=t_1, \;T_2=t_2,\;\ldots,\;T_m=t_m\leq k$
($T_i>k$, for $i>m$, because $M_k=m$). Then, by Corollary~\ref{corollary-just-connected}, for $i=1,\ldots,m$, we have
$\left\|\widetilde \Phi(t_i,t_{i-1})\right\|\leq \left(1-c\delta^{2(t_i-t_{i-1})}\right)^\frac{1}{2}$. Combining this with submultiplicativity of the spectral norm, we get:
\begin{eqnarray}
\left\|\widetilde\Phi(k,0)\right\|&=&\left\|\widetilde\Phi(k,t_m) \widetilde\Phi(t_m, t_{m-1})\cdots \widetilde \Phi(t_1,0)\right\|\\
&\leq& \left\|  \widetilde\Phi(k,t_m)\right\|   \left\| \widetilde\Phi(t_m, t_{m-1})\right\|\cdots
\left\|\widetilde \Phi(t_1,0)\right\|\nonumber\\
&\leq& \prod_{i=1}^m \left(1-c\,\delta^{2(t_i-t_{i-1})}\right)^\frac{1}{2}\nonumber.
\end{eqnarray}
  To show~\eqref{eq-lemma-condition-on-Mk=m}, we find the worst case of $t_i$'s, $i=1,\ldots,m$ by solving the following problem:
\begin{eqnarray}
\max_{\left\{\sum_{i=1}^m \Delta_i\leq k,
\:\Delta_i\geq 1\right\}}\prod_{i=1}^m \left(1-c\,\delta^{2\Delta_i}\right)&=&
\max_{\left\{\sum_{i=1}^m \beta_i\leq 1, \beta_i\in\{\frac{1}{k},\frac{2}{k},\dots,1\} \right\}}
\prod_{i=1}^m \left(1-c\,\delta^{2 \beta_i k}\right) \nonumber\\
&\leq & \max_{\left\{\sum_{i=1}^m \beta_i\leq 1, \beta_i\geq 0 \right\}}
\prod_{i=1}^m \left(1-c\,\delta^{2 \beta_i k}\right)
\end{eqnarray}
(here $\Delta_i$ should be thought of as $t_i-t_{i-1}$). Taking the $\log$ of the cost function, we get a convex problem equivalent to the original one (it can be shown that the cost function is concave). The maximum is achieved for $\beta_i=\frac{1}{m}$, $i=1,\ldots,m$. This completes the proof of Lemma~\ref{lemma-Mk=m-bound}.
\end{proof}

Lemma~\ref{lemma-Mk=m-bound} provides a bound on the norm of the ``error'' matrix $\widetilde \Phi(k,0)$ in terms of the number of improvements $M_k$ up to time $k$. Intuitively, if $M_k$ is high enough relative to $k$, then the norm of $\widetilde \Phi(k,0)$ cannot stay above $\epsilon$ as $k$ increases (to see this, just take $M_k=k$ in eq.~\eqref{eq-lemma-condition-on-Mk=m}). We show that this is indeed true for all random sequences $G_1,G_2,\ldots$ for which $M_k=\alpha k$ or higher, for any choice of $\alpha \in (0,1]$; this result is stated in Lemma~\ref{lemma-alpha-bounds}, part~\ref{lemma-case-above-alpha-k}. On the other hand, if the number of improvements is less than $\alpha k$, then there are at least $k-\alpha k$ available slots in the graph sequence in which the graphs from the maximal collection can appear. This yields, in a crude approximation, the probability of $p_{\max}^{k-\alpha k}$ for the event $M_k\leq \alpha k$; part~\ref{lemma-case-above-alpha-k} of Lemma~\ref{lemma-alpha-bounds} gives the exact bound on this probability in terms of $\alpha$. We next state Lemma~\ref{lemma-alpha-bounds}.

\begin{lemma}
\label{lemma-alpha-bounds}

Consider the sequence of events $\left\{ M_k\geq \alpha k \right\}$, where $\alpha \in(0,1]$, $k=1,2,\ldots$.
For every $\alpha,\epsilon \in (0,1]$:
\begin{enumerate}
\item \label{lemma-case-above-alpha-k}
There exists sufficiently large $k_0=k_0(\alpha,\epsilon)$ such that
\begin{equation}
\mathbb P \left(\left\|\widetilde \Phi(k,0)\right\|\geq \epsilon,\; M_k \geq \alpha k \right)=0,\;\;\;\; \forall k\geq k_0(\alpha,\epsilon)
\end{equation}
\item \label{lemma-case-below-alpha-k}
\begin{equation}
\limsup_{k\rightarrow \infty} \frac{1}{k}\log \mathbb P \left(\left\|\widetilde \Phi(k,0)\right\|\geq \epsilon,\; M_k < \alpha k \right)\leq
-\alpha \log \alpha +\alpha \log |\Pi^\star(\mathcal G)| +(1-\alpha) \log p_{\max}.
\end{equation}
\end{enumerate}
\end{lemma}
\begin{proof} To prove \ref{lemma-case-above-alpha-k}, we first note that, by Lemma~\ref{lemma-Mk=m-bound} we have:
\begin{equation}
\label{eq-bound-event-alpha-k}
\left\{\left\|\widetilde \Phi(k,0)\right\|\geq \epsilon\right\}\subseteq \left\{\left(1-c\delta^{2\frac{k}{M_k}}\right)^\frac{M_k}{2}\geq \epsilon \right\}.
\end{equation}
This gives for fixed $\alpha$, $\epsilon$:
\begin{align}
\label{eq-sum-of-prob}
\mathbb P \left(\left\|\widetilde \Phi(k,0)\right\|\geq \epsilon,\; M_k\geq\alpha k \right)&\leq
\mathbb P \left(\left(1-c\delta^{2\frac{k}{M_k}}\right)^\frac{M_k}{2}\geq \epsilon,\; M_k\geq\alpha k \right)\nonumber\\
&=\sum_{m=\left \lceil \alpha k \right \rceil}^k \mathbb P \left(\left(1-c\delta^{2\frac{k}{M_k}}\right)^\frac{M_k}{2}\geq \epsilon,\; M_k = m \right)\nonumber\\
&=\sum_{m=\left \lceil \alpha k \right \rceil}^k \mathbb P \left(g(k,m)\geq \frac{\log\epsilon}{k},\; M_k = m \right),
\end{align}
where $g(k,m):=\frac{m}{2k}\log\left(1-c\delta^{2\frac{k}{m}}\right)$, for $m>0$, and $\left \lceil x \right \rceil $ denotes the smallest integer not less than $x$.
For fixed $k$, each of the probabilities in the sum above is equal to $0$ for those $m$ such that $g(k,m)<\frac{\log\epsilon}{k}$. This yields:
\begin{equation}
\label{eq-bound-by-the-sum-of-switches}
\sum_{m=\left \lceil \alpha k \right \rceil}^k \mathbb P \left(g(m)\geq \frac{\log\epsilon}{k},\; M_k = m \right)
\leq \sum_{m= \left \lceil \alpha k \right \rceil }^k s(k,m),
\end{equation}
where $s(k,m)$ is the switch function defined by:
\begin{equation*}
s(k,m):=\left\{\begin{array}{ll} 0,&\;\mathrm{if}\; g(k,m)<\frac{\log\epsilon}{k}\\
1,&\;\mathrm{otherwise}  \end{array}\right.
\end{equation*}
Also, as $g(k,\cdot)$ is, for fixed $k$, decreasing in $m$, it follows that $s(k,m)\leq s(k,\alpha k )$ for $m\geq \alpha k$. Combining this with eqs.~\eqref{eq-sum-of-prob} and~\eqref{eq-bound-by-the-sum-of-switches}, we get:
\begin{equation*}
\label{eq-bound-by-the-largest-switch}
\mathbb P \left(\left\|\widetilde \Phi(k,0)\right\|\geq \epsilon,\; M_k\geq\alpha k \right)
\leq (k-\left\lceil \alpha k\right\rceil +1)s(k,\alpha k).
\end{equation*}
We now show that $s(k,\alpha k)$ will eventually become $0$, as $k$ increases, which would yield part~\ref{lemma-case-above-alpha-k} of Lemma~\ref{lemma-alpha-bounds}. To show this, we observe that $g$ has a constant negative value at $(k,\alpha k)$:
\begin{equation*}
g(k,\alpha k)=\frac{\alpha}{2} \log \left(1-c\delta^{\frac{2}{\alpha}}\right).
\end{equation*}
Since $\frac{1}{k}\log\epsilon\rightarrow 0$, as $k\rightarrow \infty$, there exists $k_0=k_0(\alpha,\epsilon)$ such that $g(k,\alpha k)<\frac{1}{k}\log\epsilon$,  for every $k\geq k_0$. Thus, $s( k,\alpha k)=0$ for every $k\geq k_0$. This completes the proof of part~\ref{lemma-case-above-alpha-k}.

To prove part~\ref{lemma-case-below-alpha-k}, we observe that
\begin{equation}
\label{eq-first-observation}
\mathbb P \left(\left\|\widetilde \Phi(k,0)\right\|\geq \epsilon,\; M_k < \alpha k\right)\leq \mathbb P \left( M_k < \alpha k\right)
=\sum_{m= 0}^{\left \lceil \alpha k \right \rceil -1}\mathbb P \left( M_k =m\right).
\end{equation}
Recalling the definition of $M_k$, we have $\left\{M_k=m \right\}=\left\{ T_m\leq k,\; T_{m+1}> k\right\}$, for $m\geq 0$; this, by further considering all possible realizations of $T_i$, $i\leq m$, yields
\begin{align}
\label{eq-Mk-eqs-m}
\mathbb P \left( M_k =m\right)=\sum_{1\leq t_1\leq\ldots\leq t_m\leq k}
\mathbb P \left(T_i=t_i, \mathrm{for}\; 1\leq i\leq m,\:T_{m+1}>k\right),
\end{align}
where the summation is over all possible realizations $T_i=t_i$, $i=1,\ldots,m$. Next, we remark that, by definition of stopping times $T_i$, supergraph $\Gamma(T_i-1,T_{i-1})$ is disconnected with probability $1$, for $i\leq m$ ($T_i$ is defined as the first time $t$ after time $T_{i-1}$ when the supergraph $\Gamma(t,T_{i-1})$ becomes connected); similarly, if $T_{m+1}>k$, then $\Gamma(k,T_m)$ is disconnected. Fixing the realizations $T_i=t_i$, $i\leq m$, this implies
\begin{align}
\mathbb P \left(T_i=t_i, \mathrm{for}\, i\leq m,\:T_{m+1}>k\right)
&\leq  \mathbb P\left(\Gamma(t_i-1,t_{i-1})\,\mathrm{is\,disconnected,\,for}\,i\leq m+1\right) \nonumber \\
& = \prod_{i=1}^{m+1} \mathbb P(\Gamma(t_i-1,t_{i-1})\,\mathrm{is\,disconnected})
\label{eq-bound-fixed-times}
\end{align}
where $t_{m+1}:=k+1$ and the equality follows by the independence of the graph realizations. Recalling Observation~\ref{obs-there-exists-maximal-for-any-disc} and the definition of $p_{\max}$, we further have, for $i\leq m+1$,
\begin{align*}
\mathbb P(\Gamma(t_i-1,t_{i-1})\,\mathrm{is\,disconnected})=
\mathbb P(\bigcup_{\mathcal H\in\Pi^\star(\mathcal G)} \left\{G_t \in \mathcal H, t_{i-1}<t<t_i\right\})
\leq |\Pi^\star(\mathcal G)| p_{\max}^{t_i-t_{i-1}-1}
\end{align*}
which, combined with~\eqref{eq-bound-fixed-times}, yields:
\begin{align}
\label{eq-bound-fixed-times-2}
\mathbb P \left(T_i=t_i, \mathrm{for}\, i\leq m,\:T_{m+1}>k\right)\leq |\Pi^\star(\mathcal G)|^{m+1} p_{\max}^{k-m}.
\end{align}
The bound in~\eqref{eq-bound-fixed-times-2} holds for any realization $T_i=t_i$, $1\leq t_1\leq \ldots \leq t_m\leq k$, of the first $m$ stopping times. Since the number of these realizations is ${k \choose m}\leq \left(\frac{ke}{m}\right)^m$ (see eq.~\eqref{eq-Mk-eqs-m}), we obtain the following bound for the probability of the event $M_k=m$, where $m\leq \left \lceil \alpha k \right \rceil -1$:
\begin{equation}
\label{eq-bound-on-Mk-eqs-m}
\mathbb P\left(M_k=m\right)\leq \left(\frac{ke}{m}\right)^m|\Pi^\star(\mathcal G)|^{m+1} p_{\max}^{k-m}.
\end{equation}
Finally, as function $h(m):=\left(\frac{ke}{m}\right)^m|\Pi^\star(\mathcal G)|^{m+1} p_{\max}^{k-m}$, that upper bounds the probability of the event $M_k=m$, is increasing for $m\leq k$, combining~\eqref{eq-first-observation} and~\eqref{eq-bound-on-Mk-eqs-m}, we get
\begin{equation}
\mathbb P \left(\left\|\widetilde \Phi(k,0)\right\|\geq \epsilon,\; M_k < \alpha k\right)\leq
\sum_{m= 0}^{\left \lceil \alpha k \right \rceil -1} h(m)\leq  \left \lceil \alpha k \right \rceil
\left(\frac{ke}{\left \lceil \alpha k \right \rceil -1}\right)^{\left \lceil \alpha k \right \rceil -1}
|\Pi^\star(\mathcal G)|^{\left \lceil \alpha k \right \rceil} p_{\max}^{k-(\left \lceil \alpha k \right \rceil -1)}.
\end{equation}
Taking the $\log$ and dividing by $k$, and taking the $\limsup_{k\rightarrow \infty}$ yields part~\ref{lemma-case-below-alpha-k} of Lemma~\ref{lemma-alpha-bounds}:
\begin{equation}
\limsup_{k\rightarrow \infty} \frac{1}{k}\log \mathbb P \left(\left\|\widetilde \Phi(k,0)\right\|\geq \epsilon,\; M_k < \alpha k \right)\leq
\alpha \log \frac {e}{\alpha} +\alpha \log |\Pi^\star(\mathcal G)| +(1-\alpha) \log p_{\max}.
\end{equation}
\end{proof}
To complete the proof of the upper bound~\eqref{eqn-upper-bound},
it remains to observe the following:
\begin{eqnarray*}
\mathbb P \left(\left\|\widetilde \Phi(k,0)\right\|\geq \epsilon\right)&=& \mathbb P \left(\left\|\widetilde \Phi(k,0)\right\|\geq \epsilon,\; M_k < \alpha k \right)+ \mathbb P \left(\left\|\widetilde \Phi(k,0)\right\|\geq \epsilon,\; M_k \geq \alpha k \right)\\
&=& \mathbb P \left(\left\|\widetilde \Phi(k,0)\right\|\geq \epsilon,\; M_k < \alpha k \right),\,\mathrm{for}\,k\geq k_0(\alpha,\epsilon),
\end{eqnarray*}
where the last equality follows by part~\ref{lemma-case-above-alpha-k} of Lemma~\ref{lemma-alpha-bounds}. Thus,
\begin{align}
\label{eq-final-step}
\limsup_{k\rightarrow \infty} \frac{1}{k}\log \mathbb P \left(\left\|\widetilde \Phi(k,0)\right\|\geq \epsilon \right) &= \limsup_{k\rightarrow \infty} \frac{1}{k}\log \mathbb P \left(\left\|\widetilde \Phi(k,0)\right\|\geq \epsilon,\; M_k < \alpha k \right)\nonumber \\
& \leq -\alpha \log \alpha +\alpha \log |\Pi^\star(\mathcal G)| +(1-\alpha) \log p_{\max}.
\end{align}
Since, by part~\ref{lemma-case-below-alpha-k} of Lemma~\ref{lemma-alpha-bounds}, inequality~\eqref{eq-final-step} holds for every $\alpha \in (0,1]$, taking the $\inf_{\alpha \in (0,1]}$ yields the upper bound~\eqref{eqn-upper-bound}. This completes
the proof of Theorem \ref{theorem-main} for the case when $\Pi^\star(\mathcal G)$ is nonempty. We now consider the case when $\Pi^\star(\mathcal G)=\emptyset$. In this case each realization of $G_t$ is connected (otherwise, $\Pi(\mathcal G)$ would contain at least this disconnected realization). Applying Corollary~\ref{corollary-just-connected} to successive graph realizations (i.e., for $s=t+1$) we get that
 \[\left\|\widetilde \Phi(k,0)\right\|\leq \left(1-c\delta^2\right)^{\frac{k}{2}}.\]
 For any $\epsilon >0$, there will exists $k_1=k_1(\epsilon)$ such that the right hand side is smaller than $\epsilon$ for all $k\geq k_1$. This implies that, for any $k\geq k_1$, the norm $\left\|\widetilde \Phi(k,0)\right\|$ is smaller than $\epsilon$ with probability $1$, thus yielding the rate $I=\infty$ in Theorem~\ref{theorem-main} for the case when $\Pi^\star(\mathcal G)=\emptyset$. This completes the proof of Theorem~\ref{theorem-main}.

\section{Computation of the exponential rate of consensus via min-cut: Gossip and link failure models}
\label{sec-Examples}
Motivated by the applications of averaging in sensor networks and distributed dynamical systems, we consider two frequently used types of random averaging models: gossip and link failure models. For a generic graph $G=(V,E)$, we show that $p_{\max}$ for both models can be found by solving an instance of a  min-cut problem over the same graph $G$. The corresponding link costs are simple functions of the link occurrence probabilities. In this section, we detail the relation between the min-cut problem and computation of $p_{\max}$.

We now state Lemma~\ref{lemma-max-over-disc-collections} on the computation of $p_{\max}$ that holds for the general random graph process that later will help us to calculate $p_{\max}$ for the gossip and link failure models. Lemma~\ref{lemma-max-over-disc-collections} assures that $p_{\max}$ can be found by relaxing the search space from $\Pi^\star(\mathcal G)$ -- the set of maximally disconnected collections, to $\Pi(\mathcal G)$ -- the set of all disconnected collections.
\vspace{- 1.5mm}
\begin{lemma}
\label{lemma-max-over-disc-collections}
\vspace{-1mm}
\begin{equation}
\label{eq-max-over-disc-collections}
p_{\max}=\max_{\mathcal H \in \Pi(\mathcal G)} p_{\mathcal H}
\end{equation}
\end{lemma}
\vspace{-2mm}
\begin{proof}
Since $\Pi^\star(\mathcal G)\subseteq \Pi(\mathcal G)$, to show~\eqref{eq-max-over-disc-collections} it suffices to show that for any $\mathcal H \in \Pi(\mathcal G)$ there exists $\mathcal H' \in \Pi^\star(\mathcal G)$ such that $p_{\mathcal H'}\geq p_{\mathcal H}$. To this end, pick arbitrary $\mathcal H \in \Pi(\mathcal G)$ and recall Observation~\ref{obs-there-exists-maximal-for-any-disc}. Then, there exists $\mathcal H'\in \Pi^\star(\mathcal G)$ such $\mathcal H\subseteq \mathcal H'$, which implies that
\begin{equation*}
p_{\mathcal H}=\sum_{G\in \mathcal H} \mathbb P\left(G_t=G\right)\leq \sum_{G\in \mathcal H'} \mathbb P\left(G_t=G\right)=p_{\mathcal H'}.
\end{equation*}
and proves~\eqref{eq-max-over-disc-collections}.
\vspace{-1mm}
\end{proof}
Before calculating the rate $I$ for gossip and link failure models,
we explain the minimum cut (min-cut) problem.

\mypar{Minimum cut (min-cut) problem}
Given an undirected weighted graph $G=(V,E,C)$ where $V$ is the
set of $N$ nodes, $E$ is the set of edges,
and $C=[c_{ij}]$ is the $N \times N$ matrix
of the edge nonnegative costs; by convention, we set $c_{ii}=0,$ for all $i$,
 and $c_{ij}=0$, for $\{i,j\} \notin E.$
  The min-cut problem is to find the sub-set
  of edges $E^\prime$ such that $G^\prime=(V,E \setminus E^\prime)$
   is disconnected and the sum $\sum_{\{i,j\}\in E^\prime}c_{ij}$
    is the minimal possible; we denote this minimal value, also
    referred to as the connectivity, by $\mathrm{mincut}(V,E,C)$.
    The min-cut problem
    is easy to solve, and there exist efficient
    algorithms to solve it, e.g.,~\cite{SimpleMinCut,compactingcuts}.

\vspace{-2mm}
\subsection{Gossip model}
\label{subsec-Gossip}
Consider the network of $N$ nodes, collected in the set $V$ and with the set $E\subseteq {V \choose 2}$ defining communication links between the nodes, such that if $\{i,j\}\in E$ then nodes $i$,$j\in V$ can communicate.
In the gossip algorithm, only one link $\{i,j\}\in E$ is active at a time. Let $p_{ij}$ be the probability of occurrence of link $\{i,j\}\in E$:
\vspace{-1.5mm}
\begin{equation}
p_{ij}=\mathbb P\left(G_t=(V,\{i,j\})\right).
\end{equation}
\vspace{-1.5mm}
We note that $\sum_{\{i,j\}\in E}p_{ij}=1$.
\begin{lemma}
\label{lemma-gossip-mincut}
Consider a gossip model on a graph $G=(V,E)$ with link probabilities $p_{ij}$, $\{i,j\}\in E$. Construct a mincut problem instance with the graph $G$ and the cost assigned to link $\{i,j\}$ equal $p_{ij}$. Then:
\begin{align}
p_{\max}^{\mbox {\scriptsize Gossip}}(V,E,P)&=1-\mathrm{mincut}(V,E,P)\\
I^{\mbox {\scriptsize Gossip}}(V,E,P)&=-\log(1-\mathrm{mincut}(V,E,P)),
\end{align}
where $P$ is the symmetric matrix that collects link occurrence probabilities, $P_{ij}=p_{ij}$, $\{i,j\}\in E$, $P_{ii}=0$, for $i=1,\ldots,N$ and $P_{ij}=0$, $\{i,j\}\notin E$.
\end{lemma}

\begin{proof}
For the gossip model, the set of all possible graph realizations $\mathcal G^{\mbox {\scriptsize Gossip}}$ is the set of all one-link subgraphs of $(V,E)$:
\begin{equation}
\mathcal G^{\mbox {\scriptsize Gossip}}=\left\{(V,\{i,j\}):\{i,j\}\in E \right\}.
\end{equation}
Also, there is a one to one correspondence between the set of collections of realizable graphs and the set of subgraphs of $G$: a collection $\mathcal H\subseteq \mathcal G$ corresponds to the subgraph $H$ of $G$ if and only if $H=\Gamma(\mathcal H)$. Thus, if we assign to each link in $G$ a cost equal to $p_{ij}$, then searching over the set $\Pi(\mathcal G)$ of all disconnected collections to find the most likely one is equivalent to searching over all disconnected subgraphs of $G$ with the maximal total cost:
\begin{align}
\label{eq-max-subgraph-weight}
p_{\max}^{\mbox{\scriptsize Gossip}}&=\max_{\mathcal H \in \Pi(\mathcal G)}p_{\mathcal H}\nonumber \\
&=\max_{E'\subseteq E,\; (V,E')\;\mathrm{is\;disc.}} \sum_{\{i,j\}\in E'}p_{ij}.
\end{align}
Using the fact that $\sum_{\{i,j\}\in E}p_{ij}=1$, eq.~\eqref{eq-max-subgraph-weight} can be written as:
\begin{align}
\max_{E'\subseteq E,\; (V,E')\;\mathrm{is\;disc.}} \sum_{\{i,j\}\in E'}p_{ij}
&=\max_{F\subseteq E,\; (V,E\setminus F)\;\mathrm{is\;disc.}} 1-\sum_{\{i,j\}\in F}p_{ij}\\
&=1-\min_{F\subseteq E,\; (V,E\setminus F)\;\mathrm{is\;disc.}} \sum_{\{i,j\}\in F}p_{ij}.
\end{align}
The minimization problem in the last equation is the min-cut problem $\mathrm{mincut} (V,E,P)$.
\end{proof}

\mypar{Gossip on a regular network} We now consider a special case of the uniform gossip model
on a connected regular graph with degree $d$, $d=2,...,N-1,$ and the uniform link
occurrence probability $p:=p_{ij}=\frac{2}{N d}.$ It can be easily seen that
 the value of the min-cut is $p$ times the minimal number of edges that disconnects the graph,
 which equals $p d = 2/N$; this corresponds to cutting all the edges of a fixed node, i.e.,
 isolating a fixed node. Hence,
 \begin{align*}
 p_{\mathrm{max}} &= \mathbb{P} \left( \mathrm{node\,}i\mathrm{\,is\,isolated} \right) = 1 - 2/N\\
 I                &= -\log(1-2/N).
 \end{align*}
Note that the asymptotic rate $I$ is determined by the probability
 that a fixed node is isolated; and the rate $I$ does not depend on the degree $d.$

\vspace{-2mm}
\subsection{Link failure model}
\label{subsec-Link-failure}
Similarly as with the gossip model, we introduce a graph $G=(V,E)$ to model the communication links between the nodes. In contrast with the gossip model, the link failure model assumes that each feasible link $\{i,j\}\in E$ occurs independently from all the others links in the network. Let again $p_{ij}$ denote the probability of occurrence of link $\{i,j\}\in E$. (Remark that, due to the independence assumption, we now do not have any condition on the link occurrence probabilities $p_{ij}$.)
\begin{lemma}
\label{lemma-link-failure-mincut}
Consider a link failure model on a graph $G=(V,E)$ with link probabilities $p_{ij}$, $\{i,j\}\in E$. Construct a mincut problem instance with the graph $G$ and the cost of link $\{i,j\}$ equal to $-\log(1-p_{ij})$. Then:
\begin{align}
\label{eqn-link-failure-min-cut}
p_{\max}^{\mbox {\scriptsize Link\:fail.}}(V,E,P)&=e^{-\mathrm{mincut}(V,E,-\log(1-P))}\\
I^{\mbox {\scriptsize Link\:fail.}}(V,E,P)&=\mathrm{mincut}(V,E,-\log(1-P)),
\end{align}
where $P$ is the symmetric matrix that collects the link occurrence probabilities, $P_{ij}=p_{ij}$, $\{i,j\}\in E$, $P_{ii}=0$, for $i=1,\ldots,N$ and $P_{ij}=0$, $\{i,j\}\notin E$ and $\log X$ denotes the entry wise logarithm of a matrix $X$.
\end{lemma}

\begin{proof}
 Since the links occur independently, any subgraph $H=(V,E')$ of $G$ can occur at a given time,
 therefore yielding that the collection of realizable graphs $\mathcal G^{\mbox {\scriptsize Link\:fail.}}$ is the collection of all subgraphs of $G$:
\begin{equation}
\mathcal G^{\mbox {\scriptsize Link\:fail.}}=\left\{(V,E'): E'\in 2^E    \right\};
\end{equation}
here $2^E$ denotes the power set of $E$, i.e., the collection of all possible subsets of the set of feasible \nolinebreak links~$E$.

This implies that for any fixed set $F\subseteq E$ of edges that disconnect $G=(V,E)$ we can find a disconnected collection $\mathcal H\subseteq \mathcal G$ such that $\Gamma (\mathcal H)=(V,E\setminus F)$ (recall that $\Gamma(\mathcal H)$ is the minimal supergraph of all the graphs contained in $\mathcal H$). On the other hand, any disconnected collection will map by $\Gamma$ to one disconnected subgraph of $G$. Therefore, in order to find $p_{\max}^{\mbox {\scriptsize Link\:fail.}}$ we can split the search over disconnected collections $\mathcal H$ as follows:
\begin{align}
\label{eq-split-the-search}
p_{\max}^{\mbox {\scriptsize Link\:fail.}}&=\max_{\mathcal H \subseteq \mathcal G\: \Gamma(\mathcal H)\mathrm{is\:disc.}}  p_{\mathcal H}\nonumber\\
&=\max_{F\subseteq E,\; F\:\mathrm{disconnects}\:(V,E)}\max_{\mathcal H \subseteq \mathcal G\: \Gamma(\mathcal H)=(V,E\setminus F)} p_{\mathcal H}.
\end{align}
Next, we fix a disconnecting set of edges $F\subseteq E$ and consider all $\mathcal H \subseteq \mathcal G$ such that $\Gamma(\mathcal H)=(V,E\setminus F)$. We claim that, among all such collections, the one with maximal probability is $\mathcal H_F:=\left\{(V,E'): E'\subseteq E\setminus F\right\}$. To show this, we observe that if $H=(V,E')\in \mathcal H$, then $E^\prime \cap F=\emptyset$, thus implying:
\begin{align*}
p_{\mathcal H}=\sum_{H\in \mathcal H} P\left( G_t= H \right)\leq \sum_{H=(V,E'):E^\prime\subseteq E^\prime\cap F=\emptyset} P\left( G_t= H \right)=p_{\mathcal H_F}.
\end{align*}
Therefore, the expression in~\eqref{eq-split-the-search} simplifies to:
\begin{equation*}
\max_{F\subseteq E,\; F\:\mathrm{disconnects}\:G=(V,E)} p_{\mathcal H_F}.
\end{equation*}
We next compute $p_{\mathcal H_F}$ for given $F\subseteq E$:
\begin{align*}
p_{\mathcal H_F}
&= \mathbb P\left(E(G_t)\cap F=\emptyset\right)\\
&=\mathbb P\left(\{i,j\}\notin E(G_t),\:\mathrm{for\:all\:}\{i,j\}\in F\right)\\
&=\prod_{\{i,j\}\in F}(1-p_{ij}),
\end{align*}
where the last equality follows by the independence assumption on the link occurrence probabilities.
This implies that $p_{\max}^{\mbox {\scriptsize Link\:fail.}}$ can be computed by
\begin{align*}
p_{\max}^{\mbox {\scriptsize Link\:fail.}}&=\max_{F\subseteq E,\; F\:\mathrm{disconnects}\:G=(V,E)}\prod_{\{i,j\}\in F}(1-p_{ij})\\
&=e^{-\min_{F\subseteq E,\; F\:\mathrm{disconnects}\:(V,E)}\sum_{\{i,j\}\in F}-\log(1-p_{ij})}\\
&=e^{-\min_{F\subseteq E,\; F\:\mathrm{disconnects}\:(V,E)}-\mathrm{mincut}(V,E,-\log(1-P))}.
\end{align*}
\vspace{-2mm}
\end{proof}
\mypar{Regular graph and uniform link failures} We now consider
the special case when the underlying graph is a connected regular graph
with degree $d$, $d=2,...,N-1$, and the uniform link occurrence
probabilities $p_{ij}=p.$ It is easy to see that $p_{\mathrm{max}}$ and $I$ simplify to:
\begin{align*}
p_{\max} &=  \mathbb{P} \left( \mathrm{node\,}i\mathrm{\,is\,isolated} \right) = (1-p)^d\\
I        &=  -d \log(1-p).
\end{align*}

\section{Application: Optimal power allocation for distributed detection}
\label{section-application}
We now demonstrate the usefulness of our Theorem~\ref{theorem-main} by applying it to
consensus+innovations distributed detection in~\cite{Non-Gaussian-DD,GaussianDD}
 over networks with symmetric fading links. We summarize
 the results in the current section. We first show that the asymptotic
performance (exponential decay rate of the error probability) of
distributed detection
 explicitly depends
on the rate of consensus $|\log p_{\mathrm{max}}|$. Further, we note
that $|\log p_{\mathrm{max}}|$ is a function of the link fading
(failure) probabilities, and, consequently, of the sensors'
transmission power. We exploit this fact to formulate the optimization
problem of minimizing the transmission power subject to a lower bound
on the guaranteed detection performance; the latter
translates into the requirement that $|\log p_{\mathrm{max}}|$ exceeds a
threshold. We show that the corresponding optimization problem is
convex. Finally, we illustrate by simulation the significant gains of
the optimal transmission power allocation over the uniform
transmission power allocation.

\subsection{Consensus+innovations distributed detection}
\mypar{Detection problem} We now briefly explain the distributed
detection problem that we consider. We
consider a network of $N$ sensors that cooperate to detect an event
of interest, i.e., face a binary hypothesis test $H_1$ versus $H_0$.
Each sensor $i$, at each time step $t$, $t=1,2,...,$
performs a measurement $Y_i(t)$. We assume that the measurements are $i.i.d.$, both in time and across sensors, where under hypothesis $H_l$, $Y_i(t)$ has the density function $f_l$, $l=0,1$, for $i=1,\ldots,N$ and $t=1,2,\ldots$

\mypar{Consensus+innovations distributed detector}
To resolve between the two hypothesis, each sensor $i$ maintains over
time $k$ its local decision variable $x_{i,k}$
 and compares it with a threshold; if $x_{i,k}>0$, sensor $i$ accepts $H_1$; otherwise, it accepts $H_0$.
  Sensor $i$ updates its decision variable $x_{i,k}$ by exchanging the
decision variable
    locally with its neighbors, by computing the weighted average
    of its own and the neighbors' variables, and by incorporating its
new measurement
     through a log-likelihood ratio $L_{i,k}=\log \frac{f_1(Y_{i,k})}{f_0(Y_{i,k})}$:
\begin{eqnarray}
\label{eqn-running-cons-sensor-i}
x_{i,k} = \sum_{j \in O_{i,k}}  W_{ij,k} \left(
\frac{k-1}{k}x_{j,k-1}+\frac{1}{k}L_{j,k} \right), \:k=1,2,...,
x_{i,0} = 0.
\end{eqnarray}
Here $O_{i,k}$ is the (random) neighborhood of sensor $i$ at time $k$
(including $i$), and
$W_{ij,k}$ is the (random) averaging weight that sensor $i$ assigns to sensor $j$ at time $k$.

Let
$x_k =
(x_{1,k},x_{2,k},...,x_{N,k})^\top$ and
$L_k=(L_{1,k},...,L_{N,k})^\top$.
Also, collect the averaging weights $W_{ij,k}$ in the $N \times N$
 matrix $W_k$, where, clearly, $W_{ij,k}=0$ if the sensors
  $i$ and $j$ do not communicate at time step $k$. Then, using the
definition of $\Phi(k,t)$ at the beginning of Section \ref{Sec-Main}, writing
\eqref{eqn-running-cons-sensor-i}
 in matrix form, and unwinding the recursion, we get:
\begin{equation}
\label{alg-unwinded}
x_k = \frac{1}{k} \sum_{t=1}^{k} \Phi(k,t-1) L_t,\,\,k=1,2,...
\end{equation}
Equation \eqref{alg-unwinded} shows the significance of the matrices
$\Phi(k,t)$ to the distributed detection performance, and, in particular,
on the significance of how much the matrices
$\Phi(k,t)$ are close to $J$. Indeed,
when $\Phi(k,t) = J$,
the contribution of $L_t$ to $x_{i,k}$ is $[\Phi(k,t) L_t]_i=\frac{1}{N}\sum_{i=1}^N L_{i,t}$,
and hence sensor $i$ effectively uses the local likelihood ratios of
all
the sensors. In the other extreme,
 when $\Phi(k,t) = I$, $[\Phi(k,t) L_t]_i = L_{i,t}$, and hence sensor $i$
 effectively uses only its own likelihood ratio. In fact,
 it can be shown that, when $I$ exceeds a certain threshold,
 then the asymptotic performance (the exponential decay rate
  of the error probability) at each sensor $i$ is optimal, i.e.,
  equal to the exponential decay rate
  of the best centralized detector. Specifically,
  the optimality threshold
  depends on the sensor observations
  distributions $f_1$ and $f_0$ and is given by\footnote{See \cite{Non-Gaussian-DD}
  for the precise expression of the threshold.} (see also Figure~\ref{Fig-error-exponent}):
  \begin{equation}
  \label{eqn-rate-threshold}
  I \geq I^\star\left( f_1,f_0,N\right).
  \end{equation}

  \mypar{Remark} Reference~\cite{Non-Gaussian-DD} derives a
 sufficient condition for the asymptotic optimality in terms of
 $\lambda_2(\mathbb{E}[W^2_k])$ in the form:
 $|\log \lambda_2(\mathbb{E}[W^2_k])| \geq I^\star\left( f_1,f_0,N\right)$, based
 on the inequality $\limsup_{k\rightarrow \infty}\frac{1}{k} \log \mathbb{P}(\widetilde{\Phi}(k,0)>\epsilon)
 \leq \log \lambda_2(\mathbb{E}[W^2_k])$; this inequality holds for arbitrary i.i.d. averaging models and it does not require the assumption that the positive entries of $W_k$ are bounded away from zero. The sufficient condition $|\log \lambda_2(\mathbb{E}[W^2_k])| \geq I^\star\left( f_1,f_0,N\right)$ is hence readily improved by replacing the upper bound $\log \lambda_2(\mathbb{E}[W^2_k])$ with the exact limit $-I$, whenever the matrix process satisfies Assumption~\ref{assumption}.

\subsection{Optimal transmission power allocation}
Equation \eqref{eqn-rate-threshold} says that
  there is a sufficient rate of consensus $I^\star$ such that the distributed detector
  is asymptotically optimal; a further increase of $I$ above $I^\star$ does not improve the
  exponential decay rate of the error probability. Also,
  as we have shown in subsection~\ref{subsec-Link-failure}, the rate of consensus $I$ is a function of the
  link occurrence probabilities, which are further dependent on the sensors' transmission power.
   In summary, \eqref{eqn-rate-threshold} suggests that
   there is a sufficient (minimal required) transmission power
   that achieves detection with the optimal exponential decay rate. This
   discussion motivates us to formulate the optimal power
   allocation problem of minimizing the
   total transmission power per time $k$
    subject to the optimality condition $I \geq I^\star.$ Before presenting
    the optimization problem, we detail the inter-sensor communication model.
\begin{figure}[thpb]
  \centering
  \includegraphics[scale=1, height=6cm, width=9.5cm]{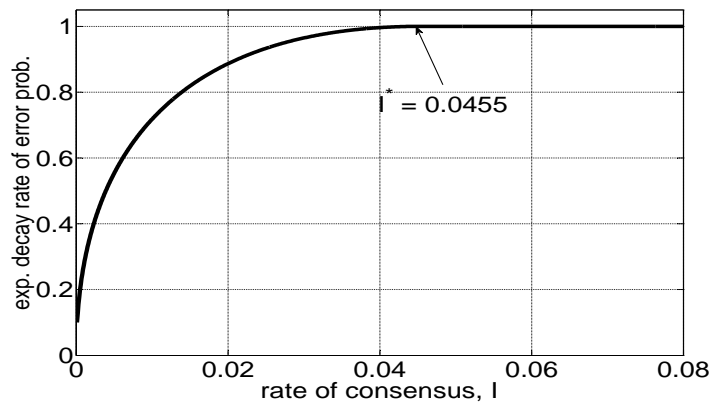}
  \caption{Lower bound on the exponential decay rate
  of the maximal error probability across sensors versus
  the rate of consensus $I$ for
  Gaussian sensor observations $f_1 \sim \mathcal{N}(m,\sigma^2)$
   and $f_0 \sim \mathcal{N}(0,\sigma^2)$.}
  \label{Fig-error-exponent}
\end{figure}

\mypar{Inter-sensor communication model}
We adopt a symmetric Rayleigh fading channel model, a model similar to the one proposed in~\cite{Scaglione-Rayleigh} (reference~\cite{Scaglione-Rayleigh} assumes asymmetric channels).
At time $k$, sensor $j$ receives from sensor $i$:
\[
y_{ij,k}=g_{ij,k}\sqrt{\frac{S_{ij}}{d_{ij}^\alpha}} x_{i,k} + n_{ij,k},
\]
where $S_{ij}$ is the transmission power that sensor $i$ uses
for transmission to sensor $j$, $g_{ij,k}$ is the channel fading
coefficient, $n_{ij,k}$
 is the zero mean additive Gaussian noise with variance $\sigma_n^2$,
$d_{ij}$ is the inter-sensor
 distance, and $\alpha$ is the path loss coefficient.
 We assume that the channels $(i,j)$ and $(j,i)$
  at time $k$ experience the same fade, i.e., $g_{ij,k}=g_{ji,k}$;
  $g_{ij,k}$ is i.i.d. in time; and $g_{ij,t}$
   and $g_{lm,s}$ are mutually independent for all $t,s.$ We adopt
 the following link failure model. Sensor
  $j$ successfully decodes the message from sensor
  $i$ (i.e., the link $(i,j)$ is online) if the signal to noise ratio
exceeds a threshold, i.e.,
  if: $\mathrm{SNR} = \frac{S_{ij}g_{ij,k}^2}{\sigma_n^2 d_{ij}^\alpha}>\tau$,
  or, equivalently, if $g_{ij,k}^2 > \frac{\tau \sigma_n^2
d_{ij}^\alpha}{S_{ij}} :=
  \frac{K_{ij}}{S_{ij}}.$ The quantity $g_{ij,k}^2$ is,
  for the Rayleigh fading channel, exponentially distributed with
  parameter 1. Hence, we arrive at the expression for the
  probability of the link $(i,j)$ being online:
  \begin{equation}
  \label{eqn-link-prob}
  P_{ij} = \mathbb{P} \left( g_{ij,k}^2>\frac{K_{ij}}{S_{ij}} \right) =
e^{-\frac{K_{ij}}{S_{ij}}}.
  \end{equation}
  We constrain the choice of transmission powers by $S_{ij}=S_{ji}$\footnote{We assumed equal noise variances $\sigma_n^2 = \mathrm{Var}(n_{ij,k})=\mathrm{Var}(n_{ji,k})$ so that $K_{ij}=K_{ji}$, which implies the constraint $S_{ij}=S_{ji}$. Our analysis easily extends to unequal noise variances, in which case we would require $\frac{K_{ij}}{S_{ij}}=\frac{K_{ji}}{S_{ji}}$; this is not considered here.}, so that
   the link
  $(i,j)$ is online if and only if the link $(j,i)$ is online, i.e.,
  the graph realizations are undirected graphs. Hence,
  the underlying communication model is the link failure model,
  with the link occurrence probabilities $P_{ij}$ in
\eqref{eqn-link-prob} that are dependent on the transmission powers
$S_{ij}$.

With this model, the rate of consensus $I$ is given by~\eqref{eqn-link-failure-min-cut}, where
the weight $c_{ij}$ associated with link $(i,j)$ is:
\[
c_{ij}(S_{ij}) =  - \log \left( 1 - e^{-K_{ij}/S_{ij}}   \right).
\]
We denote by $\{S_{ij}\}$
 the set of all powers $S_{ij}$, $\{i,j\} \in E.$
%
\begin{lemma}
\label{corollary-convex-mincut}
The function $I \left(  \{S_{ij}\} \right) =
\mathrm{mincut}(V,E,C)$, with $c_{ij} = -\log (1 - e^{-K_{ij}/S_{ij}})$, for $\{i,j\} \in E$,
 and $c_{ij}=0$ else, is concave.
\end{lemma}
\begin{proof}
Note that the function $I \left(  \{S_{ij}\} \right) =
\mathrm{mincut}(V,E,C)$ can be expressed as
\[
\min_{E^\prime \subset E:\, G^\prime = (V,E^\prime)\mathrm{\,is\,disconnected}}
\sum_{\{i,j\} \in E \setminus E^\prime } c_{ij}(S_{ij}).
\]
On the other hand, $c_{ij}(S_{ij})$ is concave in $S_{ij}$ for $S_{ij} \geq 0$, which can be shown by computing the second derivative and noting that it is non-positive. Hence, $I \left(  \{S_{ij}\} \right)$ is a pointwise minimum of concave functions, and thus it is concave.
\end{proof}
\mypar{Power allocation problem formulation}
We now formulate the optimal power
allocation problem as the problem of minimizing the total transmission
power used at time $k$, $2 \sum_{\{i,j\} \in E} S_{ij}$, so that the
distributed detector achieves asymptotic optimality. This translates into the following optimization problem:

\begin{equation}
\begin{array}[+]{ll}
\mbox{minimize} & \sum_{\{i,j\} \in E} S_{ij}\\
\mbox{subject to} & I \left( \{S_{ij}\}\right) \geq I^\star.
\end{array}.
\label{eqn-opt-problem}
\end{equation}

The cost function in \eqref{eqn-opt-problem} is
linear, and hence convex. Also, the constraint set $\left\{
\{S_{ij}\}:\,\,  I \left( \{S_{ij}\}\right) \geq I^\star\right\}=
\left\{ \{S_{ij}\}:\,\,   - I \left( \{S_{ij}\}\right) \leq  -I^\star\right\}$ is
 convex, as a sub level set of the convex function $ - I \left(
\{S_{ij}\}\right) $. (See
 Lemma \ref{corollary-convex-mincut}.) Hence, we have just
 proved the following Lemma.
 \begin{lemma}
 \label{lemma-convex-problem}
 The optimization problem~\eqref{eqn-opt-problem} is convex.
 \end{lemma}
Convexity of~\eqref{eqn-opt-problem} allows us to find a globally optimal
power allocation. The next subsection demonstrates by simulation
that the optimal power allocation significantly improves the performance of distributed detection over the uniform power allocation.
\subsection{Simulation example}
We first describe the simulation setup.
We consider a geometric network with $N=14$ sensors.
We place the sensors uniformly over a unit square, and
connect those sensors whose distance $d_{ij}$ is less than
a radius. The total number of (undirected) links is $38$.
(These $38$ links are the failing links, for which
we want to allocate the transmission powers $S_{ij}$.)
 We set the coefficients $K_{ij} = 6.25 d_{ij}^\alpha,$ with $\alpha=2.$
For the averaging weights, we use Metropolis weights, i.e., if link $\{i,j\}$ is online, we assign $W_{{ij},k}=1/(1+\max \{d_{i,k},d_{j,k}\})$, where $d_{i,k}$ is the degree of node $i$ at time $k$ and $W_{{ij},k}=0$ otherwise;
also, $W_{{ii},k}=1-\sum_{j \in O_{i,k}} W_{{ij},k}$. For the sensors' measurements,
we use the Gaussian distribution $f_1 \sim \mathcal{N}(m,\sigma^2)$,
$f_0 \sim \mathcal{N}(0,\sigma^2)$, with $m=0.0447$, and $\sigma^2=1.$ The
corresponding value $I^\star=(N-1)N\frac{m^2}{8 \sigma^2} = 0.0455.$, see \cite{Non-Gaussian-DD}.

To obtain the optimal power allocation, we solve the optimization problem \eqref{eqn-opt-problem} by
applying the subgradient algorithm with
constant stepsize $\beta = 0.0001$ on the unconstrained
exact penalty reformulation of~\eqref{eqn-opt-problem}, see,
e.g.,~\cite{Urruty}, which is to minimize $\sum_{\{i,j\} \in E} S_{ij}
+ \mu \max \left\{ 0, - \mathrm{mincut}(V,E,C)+I^\star \right\}$,
where $C=[c_{ij}]$, $c_{ij} = - \log(1-e^{-K_{ij}/S_{ij}})$, for $\{i,j\}
\in E$, and zero else; and $\mu$ is the
penalty parameter that we set to $\mu=500.$
We used the MATLAB implementation~\cite{min-cut-software} of the min-cut algorithm from~\cite{SimpleMinCut}.

\mypar{Results} Figure~\ref{Fig-simulations-error-prob} plots the detection error
probability of the worst sensor $\max_{i=1,...,N} P^e_i(k)$ versus
time $k$ for the optimal power allocation $\{S_{ij}^\star\}$ (solid blue line),
and the uniform power allocation $S_{ij}=S$ across
 all links, such that the total power
 per $k$ over all links $2 \sum_{\{i,j\} \in E}S_{ij} = 2 \sum_{\{i,j\}\in E}S_{ij}^\star=:\mathcal{S}.$
 We can see that the optimal power allocation
 scheme significantly outperforms the uniform
 power allocation. For example, to achieve the
 error probability $0.1$, the optimal
 power allocation scheme requires about $550$ time steps, hence
 the total consumed power is $550 \mathcal{S}$;
 in contrast, the uniform power allocation needs
 more than $2000 \mathcal{S}$ for the same target
 error $0.1$, i.e., about four times more power. In addition, Figure~\ref{Fig-simulations-error-prob}
 plots the detection performance
 for the uniform power allocation with
 the total power per $k$ equal to
 $\mathrm{sr} \times \mathcal{S}$,
  $\mathrm{sr} = 2,3,3.4.$ We can
  see, for example, that
   the scheme with $\mathrm{sr}=3.4$ takes
    about $600$ time steps
    to achieve an error of $0.1$,
    hence requiring about
    $600 \times 3.4 \times \mathcal{S} = 2040 \mathcal{S}$.
    In summary, for the target error of $0.1$,
    our optimal power allocation saves about $75 \%$ of
    the total power over the uniform power allocation.
\begin{figure}[thpb]
  \centering
  \includegraphics[scale=1, height=6cm, width=9.5cm]{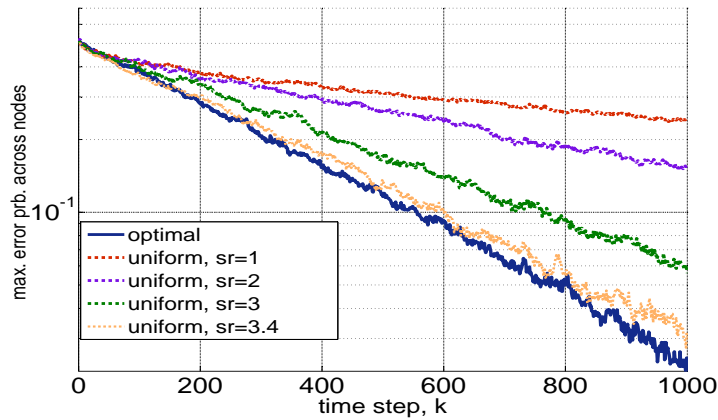}
  \caption{
  Detection error probability of the worst
  sensor versus time $k$ for the optimal and uniform
  power allocations, and different values
  of $\mathrm{sr}= \frac{\mathrm{total \,power \,per \,k\, for\,uniform\,allocation}}
  {\mathrm{total \,power \,per \,k\, for\,optimal\,allocation}}$.}
  \label{Fig-simulations-error-prob}
\end{figure}
\section{Conclusion}
\label{section-conclusion}
In this paper, we found the exact exponential decay rate $I$ of the convergence in probability for products of i.i.d. symmetric stochastic matrices $W_k$. We showed that the rate $I$ depends solely on the probabilities of the graphs that underly the matrices $W_k$. In general, calculating the rate $I$ is a combinatorial problem. However, we show that, for the two commonly used averaging models, gossip and link failure, the rate $I$ is obtained by solving an instance of the min-cut problem, and is hence easily computable. Further, for certain simple structures, we compute the rate $I$ in closed form: for gossip over a spanning tree, $I=|\log (1-p_{ij})|$, where $p_{ij}$ is the occurrence probability of the ``weakest'' link, i.e., the smallest-probability link; for both gossip and link failure models over a regular network, the rate $I=|\log p_{\mathrm{isol}}|$, where
$p_{\mathrm{isol}}$ is the probability that a node is isolated from the rest of the network at a time. Intuitively, our results show that the rate $I$ is determined by the most likely way in which the network stays disconnected over a long period of time.
Finally, we illustrated the usefulness of rate $I$ by finding a globally optimal allocation of the sensors' transmission power for consensus+innovations distributed detection.

\bibliographystyle{IEEEtran}
\bibliography{IEEEabrv,Bibliography}

\begin{thebibliography}{10}
\providecommand{\url}[1]{#1}
\csname url@samestyle\endcsname
\providecommand{\newblock}{\relax}
\providecommand{\bibinfo}[2]{#2}
\providecommand{\BIBentrySTDinterwordspacing}{\spaceskip=0pt\relax}
\providecommand{\BIBentryALTinterwordstretchfactor}{4}
\providecommand{\BIBentryALTinterwordspacing}{\spaceskip=\fontdimen2\font plus
\BIBentryALTinterwordstretchfactor\fontdimen3\font minus
  \fontdimen4\font\relax}
\providecommand{\BIBforeignlanguage}[2]{{%
\expandafter\ifx\csname l@#1\endcsname\relax
\typeout{** WARNING: IEEEtran.bst: No hyphenation pattern has been}%
\typeout{** loaded for the language `#1'. Using the pattern for}%
\typeout{** the default language instead.}%
\else
\language=\csname l@#1\endcsname
\fi
#2}}
\providecommand{\BIBdecl}{\relax}
\BIBdecl

\bibitem{SoummyaConferenceConsensus}
S.Kar and J.~M.~F. Moura, ``Distributed average consensus in sensor networks
  with random link failures,'' in \emph{ICASSP '07, IEEE International
  Conference on Acoustics, Speech and Signal Processing}, vol.~2, Pacific
  Grove, CA, April 2007, pp. {II}--1013--{II}--1016.

\bibitem{Multi-Robot}
B.~Johansson, A.~Speranzon, M.~Johansson, and K.~H. Johansson, ``On
  decentralized negotiation of optimal consensus,'' \emph{Automatica}, vol.~44,
  no.~4, pp. 1175--1179, 2008.

\bibitem{Golub-Jackson}
B.~Golub and M.~O. Jackson, ``Na\"{\i}ve learning in social networks and the
  wisdom of crowds,'' \emph{American Economic Journal: Microeconomics}, vol.~2,
  no.~1, pp. 112--149, February 2010.

\bibitem{BoydGossip}
S.~Boyd, A.~Ghosh, B.~Prabhakar, and D.~Shah, ``Randomized gossip algorithms,''
  \emph{IEEE Transactions on Information Theory}, vol.~52, no.~6, pp.
  2508--2530, June 2006.

\bibitem{SoummyaEst}
\BIBentryALTinterwordspacing
S.~Kar, J.~M.~F. Moura, and K.~Ramanan, ``Distributed parameter estimation in
  sensor networks: Nonlinear observation models and imperfect communication,''
  \emph{Accepted for publication in IEEE Transactions on Information Theory, 51
  pages}, August 2008. [Online]. Available: \url{arXiv:0809.0009v1 [cs.MA]}
\BIBentrySTDinterwordspacing

\bibitem{JadbabaieErgodic}
A.~Tahbaz-Salehi and A.~Jadbabaie, ``Consensus over ergodic stationary graph
  processes,'' \emph{IEEE Transactions on Automatic Control}, vol.~55, no.~1,
  pp. 225--230, January 2010.

\bibitem{jadbabaie_on_consensus}
------, ``On consensus over random networks,'' in \emph{44th Annual Allerton
  Conference on Communication, Control, and Computing}, Allerton House,
  Illinois, USA, September 2006, pp. 1315--1321.

\bibitem{compactingcuts}
R.~D. Carr, G.~Konjevod, G.~Little, V.~Natarajan, and O.~Parekh, ``Compacting
  cuts: a new linear formulation for minimum cut,'' \emph{ACM Transactions on
  Algorithms}, vol.~5, no.~3, July 2009, {DOI}: 10.1145/1541885.1541888.

\bibitem{tsitsiklisThesis84}
J.~N. Tsitsiklis, ``Problems in decentralized decision making and
  computation,'' Ph.D., MIT, Cambridge, MA, 1984.

\bibitem{DeGroot}
M.~H. DeGroot, ``Reaching a consensus,'' \emph{Journal of the American
  Statistical Association}, vol.~69, pp. 118--121, 1974.

\bibitem{jabaNeghbor}
A.~Jadbabaie, J.~Lin, and A.~S. Morse, ``Coordination of groups of mobile
  autonomous agents using nearest neighbor rules,'' \emph{IEEE Trans. Automat.
  Contr}, vol. AC-48, no.~6, pp. 988--1001, June 2003.

\bibitem{MurraySwitching}
R.~Olfati-Saber and R.~M. Murray, ``Consensus problems in networks of agents
  with switching topology and time-delays,'' \emph{IEEE Transactions on
  Automatic Control}, vol.~49, no.~9, pp. 1520–--1533, Sept. 2004.

\bibitem{DimakisGossip}
A.~Dimakis, A.~Sarwate, and M.~Wainwright, ``Geographic gossip: Efficient
  averaging for sensor networks,'' \emph{IEEE Transactions on Signal
  Processing}, vol.~56, no.~3, pp. 1205--1216, 2008.

\bibitem{RabbatGossip}
D.~\"{U}stebay, B.~Oreshkin, M.~Coates, and M.~Rabbat, ``Greedy gossip with
  eavesdropping,'' \emph{IEEE Transactions on Signal Processing}, vol.~58,
  no.~7, pp. 3765--3776, 2010.

\bibitem{dimakiskarmourarabbatscaglione-2010}
A.~G. Dimakis, S.~Kar, J.~M.~F. Moura, M.~G. Rabbat, and A.~Scaglione, ``Gossip
  algorithms for distributed signal processing,'' \emph{Proceedings of the
  IEEE}, vol.~98, no.~11, pp. 1847--1864, November 2010, digital Object
  Identifier: 10.1109/JPROC.2010.2052531.

\bibitem{Consensus-Delays}
A.~Nedi\'c and A.~Ozdaglar, ``Convergence rate for consensus with delays,''
  \emph{Journal of Global Optimization}, vol.~47, no.~3, pp. 437--456, 2008.

\bibitem{Consensus-Quantization}
A.~Nedi\'c, A.~Olshevsky, A.~Ozdaglar, and J.~N. Tsitsiklis, ``On distributed
  averaging algorithms and quantization effects,'' \emph{IEEE Transactions on
  Automatic Control}, vol.~54, no.~11, pp. 2506--2517, 2009.

\bibitem{Yilin}
Y.~Mo and B.~Sinopoli, ``Communication complexity and energy efficient
  consensus algorithm,'' in \emph{2nd IFAC Workshop on Distributed Estimation
  and Control in Networked Systems}, France, Sep. 2010, {DOI}:
  10.3182/20100913-2-FR-4014.00057.

\bibitem{Consensus-Olshevsky-Tsitsiklis}
A.~Olshevsky and J.~N. Tsitsiklis, ``Convergence speed in distributed consensus
  and averaging,'' \emph{SIAM Rev.}, vol.~53, pp. 747--772, November 2011.

\bibitem{Sayed-LMS}
C.~G. Lopes and A.~H. Sayed, ``Diffusion least-mean squares over adaptive
  networks: formulation and performance analysis,'' \emph{IEEE Transactions on
  Signal Processing}, vol.~56, no.~7, pp. 3122–--3136, July 2008.

\bibitem{Stankovic-Estimation}
S.~Stankovi\'{c}, M.~S. Stankovi\'{c}, and D.~M. Stipanovi\'{c}, ``Consensus
  based overlapping decentralized estimator,'' \emph{IEEE Trans. Automatic
  Control}, vol.~54, no.~2, pp. 410--415, February 2009.

\bibitem{running-consensus-detection}
P.~Braca, S.~Marano, V.~Matta, and P.~Willet, ``Asymptotic optimality of
  running consensus in testing binary hypothesis,'' \emph{IEEE Transactions on
  Signal Processing}, vol.~58, no.~2, pp. 814--825, February 2010.

\bibitem{GaussianDD}
D.~Bajovi\'{c}, D.~Jakoveti\'{c}, J.~Xavier, B.~Sinopoli, and J.~M.~F. Moura,
  ``Distributed detection via {G}aussian running consensus: Large deviations
  asymptotic analysis,'' \emph{IEEE Transactions on Signal Processing},
  vol.~59, no.~9, pp. 4381--4396, Sep. 2011.

\bibitem{Non-Gaussian-DD}
D.~Bajovi\'{c}, D.~Jakoveti\'{c}, J.~M.~F. Moura, J.~Xavier, and B.~Sinopoli,
  ``Large deviations performance of consensus+innovations distributed detection
  with non-{G}aussian observations,'' 2011, available at:
  http://arxiv.org/abs/1111.4555.

\bibitem{BruneauJoyeMerkli}
L.~Bruneau, A.~Joye, and M.~Merkli, ``Infinite products of random matrices and
  repeated interaction dynamics,'' \emph{Annales de l'Institut Henri Poincaré,
  Probabilités et Statistiques}, vol.~46, no.~2, pp. 442--464, 2010.

\bibitem{TouriNedic}
B.~Touri and A.~Nedi\'c, ``Product of random stochastic matrices,''
  \emph{Submitted to the {A}nnals of {P}robability}, 2011, available at:
  http://arxiv.org/pdf/1110.1751v1.pdf.

\bibitem{Leizarowitz}
A.~Leizarowitz, ``On infinite products of stochastic matrices,'' \emph{Linear
  Algebra and its Applications}, vol. 168, pp. 189--219, April 1992.

\bibitem{TouriNedic-backward-products}
B.~Touri and A.~Nedi\'c, ``On backward product of stochastic matrices,'' 2011,
  available at: http://arxiv.org/abs/1102.0244.

\bibitem{DiakonisWood}
P.~Diaconis and P.~M. Wood, ``Random doubly stochastic tridiagonal matrices,''
  2011, available at:
  http://stat.stanford.edu/~cgates/PERSI/papers/TriDiag1.pdf.

\bibitem{SenetaBook}
E.~Seneta, \emph{Nonnegative Matrices and Markov Chains}.\hskip 1em plus 0.5em
  minus 0.4em\relax New York: Springer, 1981.

\bibitem{Tutubalin}
V.~N. Tutubalin, ``On limit theorems for products of random matrices,''
  \emph{Theory of {P}robability and its {A}pplications}, vol.~10, pp. 15--27,
  1965.

\bibitem{Gauivarc'hRaugi}
Y.~Guivarc'h and A.~Raugi, ``Products of random matrices: convergence
  theorems,'' \emph{Contemporary mathematics}, vol.~50, pp. 31--54, 1986.

\bibitem{LePage}
{\'E}.~{Le Page}, ``Th\'eor\`emes limites pour les produits de matrices
  al\'eatoires,'' \emph{Probability Measures on Groups (Oberwolfach, 1981),
  Lecture {N}otes in {M}athematics}, vol. 928, pp. 258--303, 1982.

\bibitem{Hennion}
H.~Hennion, ``Limit theorems for products of positive random matrices,''
  \emph{The {A}nnals of {P}robability}, vol.~25, no.~4, pp. 1545--1587, 1997.

\bibitem{Kargin}
V.~Kargin, ``Products of random matrices: dimension and growth in norm,''
  \emph{The {A}nnals of {P}robability}, vol.~20, no.~3, pp. 890–--906, 2010.

\bibitem{weight-opt}
D.~Jakoveti\'{c}, J.~Xavier, and J.~M.~F. Moura, ``Weight optimization for
  consensus algorithms with correlated switching topology,'' \emph{IEEE
  Transactions on Signal Processing}, vol.~58, no.~7, pp. 3788--3801, July
  2010.

\bibitem{allerton}
D.~Bajovi\'{c}, D.~Jakoveti\'{c}, J.~Xavier, B.~Sinopoli, and J.~M.~F. {M}oura,
  ``Distributed detection over time varying networks: large deviations
  analysis,'' in \emph{48th Allerton Conference on Communication, Control, and
  Computing}, Monticello, IL, Oct. 2010, pp. 302--309.

\bibitem{Angelia}
A.~Nedi\'{c} and A.~Ozdaglar, ``Distributed subgradient methods for multi-agent
  optimization,'' \emph{IEEE Transactions on Automatic Control}, vol.~54,
  no.~1, pp. 48--61, 2009.

\bibitem{min-connectivity}
M.~Fiedler, ``Algebraic connectivity of graphs,'' \emph{Czechoslovak
  Mathematical Journal}, vol.~23, no.~98, pp. 298--305, 1973.

\bibitem{SimpleMinCut}
M.~Stoer and F.~Wagner, ``A simple min-cut algorithm,'' \emph{Journal of the
  ACM}, vol.~44, no.~4, pp. 585--591, July 1997.

\bibitem{Scaglione-Rayleigh}
K.~Chan, A.~Swami, Q.~Zhao, and A.~Scaglione, ``Consensus algorithms over
  fading channels,'' in \emph{Proc. MILCOM 2010, Military communications
  conference}, San Jose, CA, October 2010, pp. 549--554.

\bibitem{Urruty}
J.-B. Hiriart-Urruty and C.~Lemarechal, \emph{Convex Analysis and Minimization
  Algorithms: Part 1: Fundamentals}, ser. Grundlehren der Mathematischen
  Wissenschaften, {Vols. 305 and 306}.\hskip 1em plus 0.5em minus 0.4em\relax
  {B}erlin, {G}ermany: Springer-Verlag, 1993.

\bibitem{min-cut-software}
Y.~Devir, ``Matlab m-file for the min-cut algorithm,'' 2006, available at:
  http://www.mathworks.com/matlabcentral/fileexchange/13892-a-simple-min-cut-a%
lgorithm.

\end{thebibliography}
\end{document}